\documentclass[12pt,leqno]{amsart}
\usepackage{amsmath,amssymb,amsfonts, amscd,   
pifont, amsthm,  amsxtra, latexsym, 
eufrak}
\usepackage[all]{xy}
\usepackage{here}
\usepackage{verbatim}
\usepackage{graphicx}
\usepackage{psfrag}
\usepackage{amsmath}
\usepackage{amsthm}
\usepackage{amsfonts}
\usepackage{amssymb}
\usepackage{extarrows} 
\usepackage{colortbl} 
\usepackage{pstricks,pst-arrow} 
\usepackage{fancybox}

\numberwithin{equation}{section} %

\usepackage[colorlinks=true,linkcolor=blue,citecolor=blue,urlcolor=blue]{hyperref}

\newtheorem{theorem}{Theorem}[section]

\newtheorem{proposition}[theorem]{Proposition}
\newtheorem{lemma}[theorem]{Lemma}

\theoremstyle{definition}
\newtheorem{definition}[theorem]{Definition}
\newtheorem{example}[theorem]{Example}
\newtheorem{fact}[theorem]{Fact}

\newtheorem{setting}[theorem]{Setting}

\theoremstyle{remark}

\numberwithin{equation}{section} %

\usepackage[parfill]{parskip} 

\begin{document}
\author{Toshiyuki KOBAYASHI}
\thanks{
Graduate School of Mathematical Sciences,
The University of Tokyo,
3-8-1 Komaba, Tokyo 153-8914, Japan,
and
French-Japanese Laboratory in Mathematics and its Interactions,
FJ-LMI CNRS IRL2025, Tokyo, Japan.
\email{toshi@ms.u-tokyo.ac.jp}
}

\title
[Stability of Multiplicities in Symmetry Breaking]
{Stability of Multiplicities in Symmetry Breaking: 
The $\mathfrak{sl}_2$ Case}

\begin{abstract}
This expository paper explains, in the case of $\mathfrak{sl}_2$, the ideas introduced in the preprints (arXiv:2509.17007, 2604.22262), which develop a new framework for the study of multiplicities in branching laws of representations, with particular emphasis on their dependence on representation parameters.

Taking the Lie algebra $\mathfrak{sl}_2$ as a guiding example, we show that multiplicities, which are often computed via ad hoc, case-by-case arguments, are in fact governed by universal systems of linear inequalities.
To describe these inequalities, we introduce the notion of \emph{fences}, which encode the piecewise-linear boundaries of regions in parameter space on which multiplicities remain constant.

Within this framework, we give an explicit description of how multiplicities vary as parameters move inside \emph{reduced coherent families of representations}.
Our approach applies uniformly both to finite-dimensional representations and to admissible smooth Fr\'echet representations of real reductive Lie groups, and reveals a subtle and intrinsic interplay between the parameters of a group and those of its subgroup.

As an application of the general theory, we establish stability results and explicit formulas that clarify and unify a variety of classical phenomena, including the Pieri rule, $K$-type formulas, fusion rules, and tensor products of Verma modules.
In particular, the stability of fusion multiplicities provides a concrete manifestation of the theory.
More broadly, this framework suggests a unified approach to branching multiplicities extending beyond the $\mathfrak{sl}_2$ case.
\end{abstract}
\maketitle
\setcounter{tocdepth}{1}
\tableofcontents

\section{Introduction}
This article proposes a new perspective on multiplicities in branching laws of representations, with a particular emphasis on stability phenomena.
Our primary focus is the Lie algebra $\mathfrak{sl}_2$, where these features can be illustrated through classical examples that exhibit a common underlying structure.
In a variety of settings, multiplicities are traditionally obtained through detailed case-by-case analyses; however, as we demonstrate, they are in fact governed by universal systems of inequalities that control how multiplicities change as representation parameters vary.
Identifying and organizing these inequalities in a unified way is the main objective of the present work.

We begin by presenting four representative examples that share a common structural pattern.

\begin{enumerate}
\item[(1)] \textbf{(Pieri rule)}  
For the irreducible $(m+1)$-dimensional representations $V_m$ of $SU(2)$, the irreducible constituents in the tensor product decomposition are determined by a system of inequalities (see \eqref{eqn:Pieri}), together with the parity condition $m+n \equiv k \pmod{2}$:
\[
  [V_m \otimes V_n : V_k] \neq 0
  \iff
  |m-n| \le k \le m+n.
\]

\item[(2)] \textbf{($K$-type formula for discrete series)}  
For the discrete series representation $\Pi_{\lambda}$ of $SL(2,\mathbb{R})$, the $K$-type decomposition is governed by inequalities (see Example~\ref{ex:26031730}), together with the parity condition $\lambda \equiv n+1 \pmod{2}$:
\[
[\Pi_\lambda: \mathbb{C}_n] = 1 
\iff
n \ge \lambda + 1.
\]

\item[(3)] \textbf{(Fusion rule for discrete series (\cite{Mol79, Rep79}))}  
Let $\Pi_{\lambda_i}$, $i=1,2,3$, be discrete series representations of $SL(2,\mathbb{R})$. Then
\[
[\Pi_{\lambda_1} \otimes \Pi_{\lambda_2} : \Pi_{\lambda_3}] = 1
\iff
\lambda_3 \ge \lambda_1 + \lambda_2 + 1,
\]
together with the parity condition $\lambda_1 + \lambda_2 - \lambda_3 \in 2\mathbb{Z} + 1$.

\item[(4)] \textbf{(Tensor products of Verma modules)}  
Let $M(a)$, $M(b)$, and $M(c)$ be Verma modules over $\mathfrak{sl}(2,\mathbb{C})$. Generically, the multiplicity
\[
m(a,b,c)
=
\dim \operatorname{Hom}_{\mathfrak{g}_{\mathbb{C}}}
\bigl(M(c), M(a) \otimes M(b)\bigr)
\]
is $0$ or $1$.
However, the multiplicity jumps to $2$ precisely in cases characterized by explicit inequalities (\cite[Thm.\ 9.1]{KP16}, \cite{mura24};
 see also Example~\ref{ex:26032029}), together with the parity conditions $a,b,c \in \mathbb{Z}$ and $a+b-c \in 2\mathbb{N}$:
\[
a + b + c \le -2,
\qquad
|a-b| \le -c-2.
\]
\end{enumerate}
Rather than relying on ad hoc, case-by-case computations, our goal is to uncover a universal mechanism underlying these inequalities
and to formulate it in a representation-theoretic framework. More precisely, we formulate certain universal properties in branching laws
 in terms of infinitesimal characters,
 independent of the real form or the specific representations considered.

For infinite-dimensional representations, 
the notion of multiplicity requires a more careful formulation, as it depends on the topology of the representations. Let $G$ be a real reductive Lie group. In this article, we primarily consider the category $\mathcal{M}(G)$ of admissible smooth representations of $G$ of finite length and moderate growth, realized on Fr\'echet spaces (see, e.g., \cite[Chap.~11]{Wa92}). Let $\operatorname{Irr}(G)$ denote the set of irreducible objects in $\mathcal{M}(G)$.

Let $G \supset G'$ be a pair of real reductive groups, 
and let $\Pi \in \mathcal{M}(G)$ 
and $\pi \in \mathcal{M}(G')$.  
When $G'$ is compact, we have
\[
  \dim \operatorname{Hom}_{G'}(\pi, \Pi|_{G'})
  =
  \dim \operatorname{Hom}_{G'}(\Pi|_{G'}, \pi).
\]
However, this equality does not hold in general if $G'$ is noncompact.  

In the general setting where $G'$ is not necessarily compact, we consider
\begin{equation}
\label{eq:SBO}
  \operatorname{Hom}_{G'}(\Pi|_{G'}, \pi),
\end{equation}
the space of \emph{symmetry breaking operators}, namely, the space of $G'$-homomorphisms from $\Pi$ to $\pi$, that are continuous with respect to the Fr\'echet topologies.
Its dimension is called the \emph{multiplicity} and is denoted by
\[
   \text{$[\Pi|_{G'} : \pi]$, \quad or simply, \quad$[\Pi:\pi]$.}
\]

Our primary objective is to understand how multiplicities vary as the parameters vary when both a representation $\Pi$ of $G$ and a representation $\pi$ of its subgroup $G'$ are equipped with natural parameters. Families of representations parametrized in a compatible manner are called \emph{coherent families}. Although coherent families can be considered for $G$ and $G'$ separately, multiplicities in branching laws are governed by a subtle but highly nontrivial interaction between the parameters of the two groups, as illustrated by the four examples above. This interplay makes the problem considerably more intricate than the study of coherent families for a single group.

We give an explicit description, in the case of $\mathfrak{sl}_2$ and its direct sums, of how multiplicities change under parameter shifts. We identify relations between the parameters of the two groups that control these changes; these relations are expressed as systems of inequalities, whose boundaries we term \emph{fences} (Theorems~\ref{thm:260314_mahoroba}, \ref{thm:26031537}, and \ref{thm:26032009}). We also establish stability results for multiplicities 
 (Theorems~\ref{thm:26031818} and \ref{thm:26031806}).

This problem is naturally studied in settings where multiplicities are uniformly bounded.
Pairs of complex Lie algebras with this property were classified in \cite{xktoshima},
see also \cite{K12, xk22, xsunzhu}, and are limited to 
\[
(\mathfrak{s
l}_n, \mathfrak{gl}_{n-1})
\quad \text{ and } \quad 
(\mathfrak{o}_n, \mathfrak{o}_{n-1}).
\]
The former case has been generalized in recent preprints~\cite{HKS, xks26}, and the latter in a preprint~\cite{xk26}.

This provides a new conceptual framework for understanding branching laws beyond traditional case-by-case analysis.
 A particularly transparent manifestation of this framework appears in the study of
fusion rules for $\mathfrak{sl}_2$.
More precisely, a general translation formula for symmetry breaking operators
(Theorem~\ref{thm:26031537})
serves as a special case of our main results and leads directly to stability phenomena
for fusion multiplicities.
These stability properties are illustrated in concrete terms by two explicit examples
in Section~\ref{subsec:stability_ex},
which demonstrate how multiplicities remain constant within regions of the parameter space
bounded by finitely many linear inequalities.

\vskip 1pc
\par\noindent
{\bf{Acknowledgements.}}
\newline
The author is grateful to Professor Vladimir Dobrev for his warm hospitality 
 during the 16th International Workshop on
Lie Theory and its Applications
 in Physics, 
 held in Varna, Bulgaria, during June 16--22, 2025.  
This work was partially supported by the JSPS
 under the Grant-in Aid for Scientific Research (A) 
 (JP23H00084).  

\vskip 1pc
\par\noindent
\textbf{Convention.}\enspace
We write $\mathbb{N} := \{0,1,2,\dots\}$, 
$\mathbb{N}_+ := \{1,2,\dots\}$, and 
$\mathbb{N}_- := \{-1,-2,\dots\}$.

\section{Fundamental Examples: {${\mathfrak{sl}}(2,{\mathbb{C}}) \downarrow {\mathfrak{gl}}(1,{\mathbb{C}})$}}
To illustrate the general phenomena discussed in the introduction, 
we begin with the most elementary example of branching: 
the restriction from $\mathfrak{sl}_2$ to its Cartan subalgebra $\mathfrak{gl}_1$. 
Although this case is elementary, it already exhibits a key feature of multiplicity patterns: 
the relations governing multiplicities depend only on the complexified Lie algebra 
and the infinitesimal character, independently of the choice of real form and of whether the representations are finite- or infinite-dimensional.

This example serves as a foundation for the more detailed analyses that follow. 
At the same time, it provides a concrete and computationally transparent model 
for key concepts such as the diagonal action of the Casimir element on tensor products 
and symmetry breaking operators. 
It will also serve as a model for subsequent sections and future work, 
including the study of fusion rules for $\mathfrak{sl}_2$ and their extensions to higher-rank and infinite-dimensional settings.

\medskip
\subsection{Basic Setting
 for ${\mathfrak{sl}}(2,{\mathbb{C}}) \downarrow {\mathfrak{gl}}(1,{\mathbb{C}})$}
~~~\newline
We now formulate the problem in this simplest setting:
\[
   ({\mathfrak{g}}_{\mathbb{C}},{\mathfrak{g}}_{\mathbb{C}}')
  =({\mathfrak{sl}}(2,{\mathbb{C}}),{\mathfrak{gl}}(1,{\mathbb{C}})).
\]

We consider the standard basis of ${\mathfrak{g}}_{\mathbb{C}} = {\mathfrak{sl}}(2,{\mathbb{C}})$:
\begin{equation}
\label{eqn:sl2_tri}
  H:=\begin{pmatrix} 1 & 0 \\ 0 & -1 \end{pmatrix}, \quad
  X:=\begin{pmatrix} 0 & 1 \\ 0 & 0 \end{pmatrix}, \quad
  Y:=\begin{pmatrix} 0 & 0 \\ 1 & 0 \end{pmatrix}.
\end{equation}
Then $\{H, X, Y\}$ forms an ${\mathfrak{sl}}_2$-triple; that is,
\[
   [H,X]=2X, \quad [H, Y]=-2Y, \quad [X,Y]=H.
\]

Let ${\mathfrak{j}}_{\mathbb{C}} := {\mathbb{C}}H \subset {\mathfrak{g}}_{\mathbb{C}}$ 
be a Cartan subalgebra.  
We identify ${\mathfrak{j}}_{\mathbb{C}}^{\ast}$ with ${\mathbb{C}}$ via $\lambda \mapsto \lambda(H)$, so that the unique positive root in $\Delta({\mathfrak{g}}_{\mathbb{C}}, {\mathfrak{j}}_{\mathbb{C}})$ is $2$ and $\rho = 1$.

We use the normalization of the Casimir element $\Omega$ such that 
\[
   \Omega := H^2 + 2(XY + YX) \in U({\mathfrak{g}}_{\mathbb{C}}).
\]
Then the center ${\mathfrak{Z}}({\mathfrak{g}}_{\mathbb{C}})$
of the enveloping algebra $U({\mathfrak{g}}_{\mathbb{C}})$
is a polynomial ring generated by $\Omega$.  
We normalize the Harish-Chandra isomorphism as
\begin{equation}
\label{eqn:HC_sl2}
  \operatorname{Hom}_{{\mathbb{C}}\operatorname{-alg}}
  ({\mathfrak{Z}}({\mathfrak{g}}_{\mathbb{C}}), {\mathbb{C}})
  \simeq
  {\mathfrak{j}}_{\mathbb{C}}^{\ast}/W
  \simeq
  {\mathbb{C}}/\{\pm 1\},
\end{equation}
so that $\Omega$ acts on a module $\pi$ by the scalar $\lambda^2 - 1$
if $\pi$ has the ${\mathfrak{Z}}({\mathfrak{g}}_{\mathbb{C}})$-infinitesimal character $\lambda$.

For a $G$-module $V$ and 
\[
   \tau \in \operatorname{Hom}_{{\mathbb{C}}\operatorname{-alg}}
           ({\mathfrak{Z}}({\mathfrak{g}}_{\mathbb{C}}), {\mathbb{C}}),
\] 
let $P_{\tau}(V)$ denote the $\tau$-primary component of $V$ (in the sense of generalized eigenspaces):
\begin{equation*}
  P_{\tau}(V) := \bigcup_{k=1}^{\infty} \bigcap_{z \in {\mathfrak{Z}}({\mathfrak{g}}_{\mathbb{C}})} \operatorname{Ker}(z - \tau(z))^k.
\end{equation*}
Then we have the primary decomposition of $V \in {\mathcal{M}}(G)$:
\[
   V = \bigoplus_{\tau} P_{\tau}(V),
\]
and we denote the projection onto the $\tau$-primary component by the same letter $P_{\tau} \colon V \to V$.

We shall work in the following setting:
\begin{setting}
\label{set:sl_gl}
Let $G \supset G'$ be a pair of Lie groups 
with complexified Lie algebras
\begin{equation}
\label{eqn:sl2_gl1}
   ({\mathfrak{g}}_{\mathbb{C}}, {\mathfrak{g}}_{\mathbb{C}}')
  = ({\mathfrak{sl}}(2,{\mathbb{C}}), {\mathfrak{gl}}(1,{\mathbb{C}})),
\end{equation}
and let $\{H, X, Y\} \subset {\mathfrak{g}}_{\mathbb{C}}$ be
an ${\mathfrak{sl}}_2$-triple such that ${\mathfrak{g}}_{\mathbb{C}}' = {\mathbb{C}}H$.  

We denote by ${\mathbb{C}}_{\nu}$ the one-dimensional representation 
of ${\mathfrak{g}}_{\mathbb{C}}'$ given by 
\[
   {\mathfrak{g}}_{\mathbb{C}}' \ni t H \mapsto \nu t \in {\mathbb{C}}, 
\qquad t \in {\mathbb{R}}.
\]

For simplicity, we assume that $G'$ is connected, and use the same notation ${\mathbb{C}}_{\nu}$ for the representation lifted to $G'$.
\end{setting}

\begin{example}
\label{ex:sl_KA}
Let $G = SL(2,{\mathbb{R}})$.  

\noindent
{\rm (1)}\enspace
Let $G'$ be a maximal compact subgroup $K = SO(2)$, 
and let ${\mathbb{C}}_{\nu}$ denote the character of $K$ given by
\[
  \begin{pmatrix} \cos \theta & -\sin \theta \\ \sin \theta & \cos \theta \end{pmatrix} 
  \mapsto e^{i \nu \theta}, 
  \qquad \nu \in {\mathbb{Z}}.
\]

\noindent
{\rm (2)}\enspace
Let $G' = A$, where 
\[
A := \left\{\begin{pmatrix} e^x & 0 \\ 0 & e^{-x} \end{pmatrix} : x \in {\mathbb{R}} \right\} 
\subset G = SL(2, {\mathbb{R}}),
\]
and let ${\mathbb{C}}_{\nu}$ denote the character of $A$ defined by
\[
  \begin{pmatrix} e^x & 0 \\ 0 & e^{-x} \end{pmatrix} \mapsto e^{\nu x}, 
  \qquad \nu \in {\mathbb{C}}.
\]
\end{example}
\medskip
\subsection{Translation of Symmetry Breaking via Tensoring with ${\mathbb{C}}^2$}
\label{subsec:ts_C2}
~~~\newline
Let 
\[
f_+ := \begin{pmatrix} 1 \\ 0 \end{pmatrix}, \qquad
f_- := \begin{pmatrix} 0 \\ 1 \end{pmatrix},
\]
and write 
\[
{\mathbb{C}}^2 := {\mathbb{C}} f_+ \oplus {\mathbb{C}} f_-
\]
for the natural representation of ${\mathfrak{g}}_{\mathbb{C}}={\mathfrak{sl}}(2, {\mathbb{C}})$.

We investigate how symmetry breaking operators transform under tensor product
 with the standard representation $\mathbb{C}^2$, which provides a basic instance of translation of infinitesimal characters.

Suppose that $\Pi \in {\mathcal{M}}(G)$ has infinitesimal character 
$\lambda \in \mathbb{C}$ (with respect to the Harish-Chandra isomorphism \eqref{eqn:HC_sl2}).
 Then $\Pi \otimes {\mathbb{C}}^2$ admits the following primary decomposition
\[
\Pi \otimes {\mathbb{C}}^2 \simeq 
P_{\lambda+1}(\Pi \otimes {\mathbb{C}}^2) \oplus 
P_{\lambda-1}(\Pi \otimes {\mathbb{C}}^2).
\]

For $\delta \in \{1, -1\}$, let
\begin{equation}
\label{eqn:pr_pm}
\operatorname{pr}_{\lambda \to \lambda+\delta} \colon 
\Pi \otimes {\mathbb{C}}^2 \longrightarrow 
P_{\lambda+\delta}(\Pi \otimes {\mathbb{C}}^2)
\end{equation}
denote the natural projection.

\begin{theorem}
\label{thm:260314_mahoroba}
Assume that
\begin{equation}
\label{eqn:sl2_scalar}
\text{${\mathfrak{Z}}({\mathfrak{g}}_{\mathbb{C}})$ acts on each primary component $P_{\lambda \pm 1}(\Pi \otimes {\mathbb{C}}^2)$ by scalars.}
\end{equation}
Let $\Pi \in \mathcal{M}(G)$ have a nonsingular infinitesimal character $\lambda \neq 0$.  
Let $\varepsilon \in \{+, -\}$ and $\delta \in \{1, -1\}$.  
Let $T \colon \Pi \to {\mathbb{C}}_{\nu}$ be a symmetry breaking operator.  
Then, for every $u \in \Pi$, we have
\begin{equation}
\label{eqn:26031442}
\bigl(T \otimes \operatorname{pr}_{\varepsilon}\bigr) 
\circ \operatorname{pr}_{\lambda \to \lambda + \delta} (u \otimes f_{\varepsilon})
= \frac{\lambda + \delta(1 + \varepsilon \nu)}{2 \lambda} \,(T u) \otimes f_{\varepsilon}.
\end{equation}
\end{theorem}

This identity \eqref{eqn:26031442} describes explicitly how symmetry breaking operators transform under shifts of the infinitesimal character.

We define the translation functor 
$\psi_{\lambda}^{\lambda \pm 1}$ by
\[
\psi_{\lambda}^{\lambda+\delta}(\Pi) := 
\operatorname{pr}_{\lambda \to \lambda+\delta} (\Pi \otimes \mathbb{C}^2).
\]

Under the assumption \eqref{eqn:sl2_scalar} and $\lambda \neq 0$, 
Theorem~\ref{thm:260314_mahoroba} immediately implies the following propagation property of multiplicities:

\begin{equation}
\label{eqn:up_down}
[\Pi : \mathbb{C}_{\nu}] \neq 0 
\;\;\Rightarrow\;\; 
[\psi_{\lambda}^{\lambda+\delta}(\Pi) : \mathbb{C}_{\nu+\varepsilon}] \neq 0,
\end{equation}
whenever $\lambda + \delta (1 + \varepsilon \nu) \neq 0$.

We now explain how Theorem~\ref{thm:260314_mahoroba} controls the multiplicities in a concrete setting involving finite-dimensional representations of $SU(2)$. In this case, assumption \eqref{eqn:sl2_scalar} is automatically satisfied due to the complete reducibility of representations of compact groups.

\begin{example}
\label{ex:26031506}
For every $n \in \mathbb{N}$, there exists a unique irreducible $(n+1)$-dimensional representation of $G = SU(2)$, denoted by $V_n$. It has $\mathfrak{Z}(G)$-infinitesimal character $n+1$.

We consider the weights of $V_n$ with respect to the maximal torus $G' := T \subset G$.
The irreducible decomposition (the Pieri rule)
\[
V_n \otimes \mathbb{C}^2 \simeq 
\begin{cases}
V_{n+1}, & \text{if } n = 0,\\
V_{n+1} \oplus V_{n-1}, & \text{if } n \ge 1,
\end{cases}
\]
implies that the family $\{V_n\}_{n \in \mathbb{N}}$ satisfies
\begin{equation}
\label{eqn:fd_coh}
\psi_{n}^{n+1}(V_{n-1}) \simeq V_{n}, \qquad
\psi_{n+1}^{n}(V_{n}) \simeq V_{n-1},
\end{equation}
for every $n \in \mathbb{N}_+$.
(This property is part of the definition of a {\emph{reduced coherent family}} of $G$-modules; see Definition~\ref{def:coherent} below.)

When $n = 0$, $V_0$ is the trivial one-dimensional representation, and hence
\begin{equation}
\label{eqn:trivial}
[V_0 : \mathbb{C}_{\nu}] \neq 0 \iff \nu = 0.
\end{equation}

By iterating \eqref{eqn:up_down}, we can determine the weights of $V_n$. Specifically,
\[
[V_n : \mathbb{C}_{\nu}] \neq 0 \iff \nu \in \{ n, n-2, \dots, -n \},
\]
for every $n \in \mathbb{N}$, starting from the trivial case \eqref{eqn:trivial} for $n = 0$.

\vskip 1pc
\noindent
$\Rightarrow$)\enspace Assume for some $n \in {\mathbb{N}}_+$ that
\[
  [V_{n-1} : \mathbb{C}_{\nu-1}] = 0 \quad
  \text{if } \nu \notin \{ n, n-2, \dots, -n \}.
\]
Applying \eqref{eqn:up_down} with $\delta=-1$, $\varepsilon=-$,
 and $\lambda =n+1$, we obtain
\[
[V_{n-1} : \mathbb{C}_{\nu-1}] = [\psi_{n+1}^{n}(V_{n}) : \mathbb{C}_{\nu-1}] = 0
\;\;\Rightarrow\;\; 
[V_{n} : \mathbb{C}_{\nu}] = 0
\]
if $(n+1)-(1-\nu) \neq 0$.  
Thus,
\[
[V_{n} : \mathbb{C}_{\nu}] = 0 \quad \text{if } \nu \notin \{ n, n-2, \dots, 2-n \} \cup \{-n\}.
\]

\vskip 1pc
\noindent
$\Leftarrow$)\enspace Assume for some $n \in {\mathbb{N}}_+$ that
\[
[V_{n-1} : \mathbb{C}_{\nu}] \neq 0 \quad \text{if } \nu \in \{ n-1, n-3, \dots, 1-n \}.
\]
Applying \eqref{eqn:up_down} with $\delta=1$, $\lambda = n$, we obtain
\[
[V_{n-1} : \mathbb{C}_{\nu}] \neq 0 \;\;\Rightarrow\;\;
[\psi_n^{n+1}(V_{n-1}) : \mathbb{C}_{\nu+\varepsilon}] = [V_{n} : \mathbb{C}_{\nu+\varepsilon}] \neq 0
\]
if $n+1 + \varepsilon \nu \neq 0$.  
Hence $V_{n}$ has weights
\begin{alignat*}{2}
& n-2, n-4, \dots, -n, && \quad \text{from the case } \varepsilon = -, \\
& n, n-2, \dots, 2-n, && \quad \text{from the case } \varepsilon = +.
\end{alignat*}
Thus the induction proceeds.
\end{example}

This example pertains to the finite-dimensional case. 
Although the result is classical, the method presented here is new: it reconstructs branching laws inductively from the trivial representation using translation. 
This method can also be applied to symmetry breaking for the infinite-dimensional representations.

\medskip
\subsection{Reduced Coherent Family of Representations}
\label{subsec:coherent_family}
~~~\newline
To extend the idea of Example~\ref{ex:26031506} 
to the infinite-dimensional setting, 
we recall the notion of a coherent family of representations.  
Although coherent families provide a powerful framework for studying wall-crossing phenomena, 
we shall only consider parameters lying within a fixed Weyl chamber.

In the framework of coherent continuation, 
the following result is particularly important 
(we do not recall the general definition here).  

\begin{fact}
\label{fact:coherent}
Let $G$ be a real linear Lie group 
contained in a connected complexified Lie group $G_{\mathbb{C}}$, 
and let $\Pi \in \operatorname{Irr}(G)$ have a nonsingular 
infinitesimal character.  
Then there exists a unique coherent family 
in $\operatorname{Irr}(G)$ passing through $\Pi$.  
\end{fact}

In the present paper, 
however, 
we use only one aspect of coherent families in a special case.  
More precisely, we adopt the following simplified  definition.  

\begin{definition}[Reduced Coherent Family]
\label{def:coherent}
Let $G$ be a real form of $SL(2, \mathbb{C})$, 
and let $\Pi \in \mathcal{M}(G)$ have 
infinitesimal character $\xi \in \mathbb{C}$.  
Assume that $\xi$ is nonsingular, i.e., $\xi \neq 0$.  

Define the parameter set $\Lambda$ by 
\begin{equation}
\label{eqn:dom_xi}
\Lambda \equiv \Lambda(\xi) :=
\begin{cases}
\mathbb{Z} + \xi, & \text{if } \xi \notin \mathbb{Z},\\
\mathbb{N}_+, & \text{if } \xi \in \mathbb{N}_+,\\
\mathbb{N}_-, & \text{if } \xi \in \mathbb{N}_-.
\end{cases}
\end{equation}

A family of representations 
$\{\Pi_{\lambda} \in \mathcal{M}(G) : \lambda \in \Lambda\}$ 
is said to form a {\emph{reduced coherent family}} through $\Pi$ if the following conditions are satisfied:

\begin{itemize}
\item $\Pi_{\lambda} = \Pi$ \quad if $\lambda = \xi$;
\item $\Pi_{\lambda}$ has infinitesimal character $\lambda$;
\item For every $\delta \in \{1, -1\}$, whenever $\lambda$ and $\lambda+\delta$ both lie in $\Lambda$,
\begin{equation}
\label{eqn:disc_coh}
P_{\lambda+\delta}(\Pi_{\lambda} \otimes \mathbb{C}^2) \simeq \Pi_{\lambda+\delta}, \qquad
P_{\lambda}(\Pi_{\lambda+\delta} \otimes \mathbb{C}^2) \simeq \Pi_{\lambda}.
\end{equation}
\end{itemize}
\end{definition}

\begin{example}
Let $G = SU(2)$.  
Define $\Pi_{\lambda} := V_{\lambda-1}$ for $\lambda \in \mathbb{N}_+$.  
Then $\{\Pi_{\lambda}\}$ forms a reduced coherent family, 
where $V_m$ denotes the $(m+1)$-dimensional irreducible representation.
\end{example}

\begin{example}
\label{ex:ps_sl}
Let $G = SL(2, \mathbb{R})$, 
and let $P \subset G$ be the parabolic subgroup 
consisting of upper triangular matrices.  
For $\lambda \in \mathbb{C}$ and $\kappa \in \mathbb{Z}/2\mathbb{Z}$, 
let $\mathbb{C}_{\lambda, \kappa}$ denote 
the one-dimensional representation of $P$ defined by 
\[
  \chi_{\lambda, \kappa} \colon P \to \mathbb{C}^{\times}, 
\qquad
  \begin{pmatrix} a & b \\ 0 & -a \end{pmatrix}
  \mapsto 
  |a|^{\lambda} (\operatorname{sgn} a)^{\kappa}.  
\]
We define the normalized principal series representation of $G$ by
\[
 \Pi(\lambda, \kappa) := \operatorname{Ind}_P^G(\mathbb{C}_{\lambda+1, \kappa}).  
\]
Then, for any $\xi \in \mathbb{C} \setminus \{0\}$, 
the family
\[
   \{\Pi(\xi+m, \kappa+m) : m \in \mathbb{Z}\}
\]
forms a (reduced) coherent family. 
\end{example}

\begin{example}
\label{ex:disc}
Let $G = SL(2, \mathbb{R})$.  
Then the set of holomorphic discrete series representations
\[
  \{\Pi_{\lambda} : \lambda \in \mathbb{N}_+\}
\]
forms a reduced coherent family.  
Each representation $\Pi_{\lambda}$ has infinitesimal character $\lambda$, 
and its minimal $K$-type is given by
$\mathbb{C}_{\lambda+1} \in \widehat{K}$, where $K = SO(2)$.  
By our convention, we equip $\Pi_{\lambda}$ with the Fr{\'e}chet topology
of smooth vectors, rather than the Hilbert space topology.  
\end{example}

\medskip 
\subsection{Stability Theorem of Multiplicities in Symmetry Breaking}
\label{subsec:stability}
~~~\newline
As an application of Theorem~\ref{thm:260314_mahoroba}, 
we present a stability theorem 
for multiplicities in symmetry breaking, see \eqref{eq:SBO}.  

For subsets $\Lambda \subset \mathbb{C}$ and $H \subset \mathbb{C}$, 
we set $\Xi := \Lambda \times H$ and define
\begin{align*}
 \Xi^{\uparrow} &:= \{(\lambda, \nu) \in \Xi : 
  \operatorname{Re} \nu \ge |\operatorname{Re} \lambda| + 1 \}, 
\\
 \Xi^{\rightarrow} &:= \{(\lambda, \nu) \in \Xi : 
  \operatorname{Re} \lambda \ge |\operatorname{Re} \nu| + 1 \}, 
\\
 \Xi^{\downarrow} &:= \{(\lambda, \nu) \in \Xi : 
  \operatorname{Re} \nu \le -|\operatorname{Re} \lambda| - 1 \}, 
\\
 \Xi^{\leftarrow} &:= \{(\lambda, \nu) \in \Xi : 
  \operatorname{Re} \lambda \le -|\operatorname{Re} \nu| - 1 \}.  
\end{align*}
\begin{theorem}
[Stability of Multiplicities]
\label{thm:26031818}
Let $(G, G')$ be as in Setting~\ref{set:sl_gl}, 
 and let $\{\Pi_{\lambda}: \lambda \in \Lambda(\xi)\} \subset {\mathcal{M}}(G)$
 be the reduced coherent family 
 passing through $\Pi=\Pi_{\xi} \in {\mathcal{M}}(G)$
 with nonsingular infinitesimal character $\xi \ne 0$.

Let ${\mathbb{C}}_{\eta}$ be a character of $G'$
 for some $\eta \in {\mathbb{C}}$.

Set 
\[
  \Xi:=\Lambda(\xi) \times ({\mathbb{Z}}+\eta).  
\]
\par\noindent
{\rm{(1)}}\enspace
If $\xi-\eta \not\in 2{\mathbb{Z}}+1$, 
 then 
\[
   [\Pi_{\lambda}:{\mathbb{C}}_{\nu}]=[\Pi:{\mathbb{C}}_{\eta}]
\]
 for all $(\lambda, \nu) \in \Xi$.  
\par\noindent
{\rm{(2)}}\enspace
Suppose that $\xi-\eta \in 2{\mathbb{Z}}+1$.  
Then $(\xi, \eta) \in \Xi^s$
 for some $s \in \{\uparrow, \rightarrow, \downarrow, \leftarrow\}$.  
Moreover, 
\[
   [\Pi_{\lambda}:{\mathbb{C}}_{\nu}]
   =
   [\Pi_{\xi}:{\mathbb{C}}_{\eta}]
\]
for every $(\lambda, \nu) \in \Xi^s$.  
\end{theorem}

When $\eta \in {\mathbb{Z}}$
 (for example, 
 when $G$ is a compact torus), 
we may reformulate Theorem~\ref{thm:26031818} as follows:
\begin{theorem}
[Stability of Multiplicities]
\label{thm:26031806}
Retain the setting of Theorem~\ref{thm:26031818}.  
In particular, 
 let $\Pi \in {\mathcal{M}}(G)$ have nonsingular infinitesimal character
 $\xi \in {\mathbb{C}} \setminus \{0\}$, 
 and let $\{\Pi_{\lambda}:\lambda \in \Lambda(\xi)\}$
 be the reduced coherent family passing through $\Pi$.  
Assume that $\eta \in {\mathbb{Z}}$ 
 so that the parameter set $H$ for ${\mathbb{C}}_{\nu}$
 is identified with ${\mathbb{Z}}$, 
 and set
\[
  \Xi:= \Lambda(\xi) \times {\mathbb{Z}}.  
\]
\par\noindent
{\rm{(1)}}\enspace
Suppose that $\xi \not\in {\mathbb{Z}}$.  
Then the multiplicity
\[
   [\Pi_{\lambda}:{\mathbb{C}}_{\nu}], 
\qquad 
 (\lambda, \nu) \in \Xi, 
\]
 depends only on the parity
 of $\lambda -\xi-\nu \in {\mathbb{Z}}$.  
\par\noindent
{\rm{(2)}}\enspace
Suppose that $\xi \in {\mathbb{N}}_+$.  
Then the parameter space
\[
  \Xi = \Lambda(\xi) \times {\mathbb{Z}}={\mathbb{N}}_+ \times {\mathbb{Z}}
\]
decomposes as a disjoint union 
\[
   \Xi=\Xi_{\operatorname{even}} \amalg \Xi_{\operatorname{odd}}.  
\]
according to the congruence condition 
\[
   \text{$\lambda \equiv \nu\, (\operatorname{mod} 2)$
 and $\lambda \equiv \nu+1\, (\operatorname{mod} 2)$, 
}
\]
 respectively.  
Furthermore, 
 we decompose
\[
      \Xi_{\operatorname{odd}}
     =
      \Xi_{\operatorname{odd}}^{\uparrow}
      \amalg 
      \Xi_{\operatorname{odd}}^{\rightarrow}
      \amalg
      \Xi_{\operatorname{odd}}^{\downarrow}
\]
 according to the conditions $\nu \ge \lambda+1$, 
 $\lambda \ge |\nu| +1$, and $\nu \le -\lambda-1$, 
respectively. 
Then the multiplicity $[\Pi_{\lambda}:{\mathbb{C}}_{\nu}]$ remains constant within each of the four regions
\[\Xi_{\operatorname{even}}, \quad
   \Xi_{\operatorname{odd}}^{\uparrow}, \quad
   \Xi_{\operatorname{odd}}^{\rightarrow}, \quad
   \Xi_{\operatorname{odd}}^{\downarrow}
\]
in the $(\lambda, \nu)$-plane defined above.

\end{theorem}
\begin{example}
\label{ex:su2_so2}
[Restatement of Example~\ref{ex:26031506}]
Let $(G,G')=(SU(2), {\mathbb{T}})$.  
Let $\lambda \in \mathbb{N}_+$, 
and let $\Pi_{\lambda} := V_{\lambda-1}$ be the irreducible
$\lambda$-dimensional representation of $G$. 
Then $\Xi=\mathbb N_+ \times \mathbb Z$ and
\[
  \begin{cases}
  0, & \text{if $(\lambda, \nu) \in \Xi_{\operatorname{even}} 
                     \cup \Xi_{\operatorname{odd}}^{\uparrow} 
                     \cup \Xi_{\operatorname{odd}}^{\downarrow}$},   \\
  1, & \text{if $(\lambda, \nu) \in \Xi_{\operatorname{odd}}^{\rightarrow}$}.
  \end{cases}
\]
\end{example}

The following example appears to be absent from the literature: despite the failure of general multiplicity-one results,
the multiplicity is not only uniformly bounded but identically equal to two,
independently of all parameters.

We emphasize that, in our setting, $(G,G')$ is a pair of real reductive groups
whose complexified Lie algebras are of the form
$(\mathfrak{sl}_n, \mathfrak{gl}_{n-1})$.
The general theory of uniform boundedness
(Kobayashi--Oshima~\cite{xktoshima}) still applies in this setting.
However, the multiplicity-one theorem of Sun--Zhu~\cite{xsunzhu}
does not apply, since the subgroup $A$ is connected.
\begin{example}
\label{ex:26032602}
    Let $G=SL(2, \mathbb{R})$, 
and let $\Pi(\lambda, \kappa)$ $(\lambda \in \mathbb{C}, \kappa \in {\mathbb Z}/2\mathbb Z)$
denote the principal series representation of $G$
as in Example~\ref{ex:ps_sl}. 
Let $\mathbb C_\nu$ ($\nu\in \mathbb C$) be a character of the subgroup $A\subset G$ as in Example~\ref{ex:sl_KA}.
Then 
\[
[\Pi(\lambda, \kappa):\mathbb C_\nu]=1 \quad \text{ for all } (\lambda, \kappa, \nu)
\in  \mathbb{C} \times {\mathbb Z}/2\mathbb Z \times \mathbb C.
\]
In particular, the multiplicity is independent of the parameters and no vanishing occurs.
The proof can be obtained by adapting the method of Kobayashi--Speh~\cite[Thm.~3.16]{KS15} to this special case.
\end{example}
\begin{example}
\label{ex:26031730}
Let $G=SL(2, \mathbb{R})$, 
and let $\Pi_{\lambda}$ $(\lambda \in \mathbb{N}_+)$
denote the discrete series representation of $G$
as in Example~\ref{ex:disc}.  

\par\noindent
{\rm (1)}\enspace $(G \downarrow K)$\enspace
Let $K=SO(2)$.  
Then, using only the following three $K$-types 
of $\Pi_{\lambda}$ for $\lambda = 1$:
\begin{equation}
\label{eqn:disc_1}
  [\Pi_{\lambda} : \mathbb{C}_0] = [\Pi_{\lambda} : \mathbb{C}_1] = 0, 
\qquad
  [\Pi_{\lambda} : \mathbb{C}_{\nu}] = 1,
\end{equation}
we can determine the $K$-type formula
for every $\Pi_{\lambda}$ $(\lambda \in \mathbb{N}_+)$;
that is, 
\[
     [\Pi_{\lambda} : \mathbb{C}_{\nu}]
     =
     \begin{cases}
     1, & \text{if $\nu \in \lambda+1 + 2\mathbb{N}$, } \\
     0, & \text{otherwise,}
    \end{cases}
\]
by Theorem~\ref{thm:260314_mahoroba}.
In the context of Theorem~\ref{thm:26031806},
 $\Xi= \mathbb N_+ \times \mathbb Z$ and
\[
     [\Pi_{\lambda} : \mathbb{C}_{\nu}]
     \iff (\lambda, \nu) \in \Xi_{\operatorname{odd}}^{\uparrow}.   
\]
\noindent
{\rm (2)}\enspace $(G \downarrow A)$\enspace
In contrast to the restriction $G \downarrow K$,
the following holds:
\begin{equation}
\label{eqn:26032605}
   [\Pi_{\lambda} : \mathbb{C}_{\nu}] = 1
   \qquad \text{for all $(\lambda, \nu) \in \mathbb N_+ \times \mathbb C$.}
\end{equation}
We divide the argument into two cases.

First, suppose that $\nu \notin \lambda + 2\mathbb{Z}+1$
or $|\nu|+1 > \lambda$.
Then the assertion follows from Example~\ref{ex:26032602}.
Indeed, the discrete series representation $\Pi_\lambda$
can be realized as a quotient of a principal series representation,
whose subrepresentation is finite-dimensional and, by weight considerations, does not contain $\mathbb C_\nu$ for such $\nu$ (this also follows from Example~\ref{ex:su2_so2} via Weyl's unitary trick).

Next, we consider the remaining case.
We restrict to $\Xi=\mathbb N_+ \times \mathbb Z$.
Using the special case
\[
   [\Pi_{\lambda} : \mathbb{C}_{\nu}] = 1
   \qquad \text{for $(\lambda, \nu) = (1,0)$,}
\]
and applying Theorem~\ref{thm:260314_mahoroba},
we obtain
\[
     [\Pi_{\lambda} : \mathbb{C}_{\nu}]=1
     \quad \text{for } (\lambda, \nu) \in \Xi_{\operatorname{odd}}^{\rightarrow}.
\]

Thus, in contrast to the compact case,
no vanishing occurs and the multiplicity is identically one.
\end{example}
The multiplicity also admits a natural geometric interpretation.  
We illustrate this using the $G \downarrow A$ 
case from the above example.  
For $\nu \in \mathbb{C}$, 
consider the $G$-equivariant line bundle 
\[
  \mathcal{L}_{\nu} = G \times_{A} \mathbb{C}_{\nu}
\]
over the anti-de Sitter space $G/A$.  
By Frobenius reciprocity, one has
\[
  [\Pi : \mathbb{C}_{\nu}]
  =
  \dim \operatorname{Hom}_{G}(\Pi, C^{\infty}(G/A, \mathcal{L}_{\nu})),  
\]
for every smooth admissible representation $\Pi$ of $G$
and every $\nu \in \mathbb{C}$.  
Consequently, we obtain the following result:

\begin{example}
Let $\Pi_{\lambda}$ $(\lambda \in \mathbb{N}_+)$ be the smooth representation 
 of a discrete series
representation of $G = SL(2, \mathbb{R})$.  
Then, for any character $\nu \in \mathbb{C}$, 
\[
  \dim \operatorname{Hom}_{G}(\Pi_{\lambda}, C^{\infty}(G/A, \mathcal{L}_{\nu})) = 1.  
\]
\end{example}

\medskip
\subsection{Proof of Theorem~\ref{thm:260314_mahoroba}}
~~~\newline
\par
We recall that $\Omega$ denotes the Casimir operator.  
Also, for $\delta \in \{1, -1\}$, 
let
\[
   \operatorname{pr}_{\lambda \to \lambda+\delta}
   \colon 
   \Pi \otimes \mathbb{C}^2
   \longrightarrow 
   P_{\lambda+\delta}(\Pi \otimes \mathbb{C}^2)
\]
denote the projection onto the corresponding primary component (see \eqref{eqn:pr_pm}).

Under the assumption \eqref{eqn:sl2_scalar}, 
the Casimir element $\Omega$ acts on the primary component
$P_{\lambda\pm1}(\Pi \otimes \mathbb{C}^2)$
by the scalar $(\lambda\pm1)^2 - 1 = \pm 2\lambda$.  
From this, we obtain the following lemma:

\begin{lemma}
\label{lem:26031443}
We set
\[
  \Omega_{\pm} := \Omega - \lambda^2 \pm 2\lambda.  
\]

Then we have
\[
   \Omega_{\pm} = \pm 4 \lambda \, \operatorname{pr}_{\lambda \to \lambda \pm 1}.  
\]
\end{lemma}

\begin{lemma}
\label{lem:26031430}
The diagonal action of $\Omega$ on $\Pi \otimes \mathbb{C}^2$
is given explicitly by
\[
   (\Omega - \lambda^2 - 2)(u_+ \otimes f_+ + u_- \otimes f_-)
  = 2 (H u_+ + 2 Y u_-) \otimes f_+
  + 2 (X u_+ - H u_-) \otimes f_-
\]
for all $u_+, u_- \in \Pi$.  
\end{lemma}

\begin{proof}
By the Leibniz rule, for any $u \in \Pi$ and $f \in \mathbb{C}^2$, we have
\[
   \Omega(u \otimes f)
   =
   (\Omega u) \otimes f
   +
   u \otimes \Omega f
   + 2(Hu \otimes Hf + 2 Xu \otimes Yf + 2 Yu \otimes Xf).
\]
Since the Casimir element $\Omega$ acts on $\Pi$ and $\mathbb{C}^2$
as scalar multiplication by $\lambda^2 - 1$ and $3$, respectively, 
we obtain
\[
  (\Omega - \lambda^2 - 2)(u \otimes f)
   =
   2 H u \otimes Hf
   + 4 (X u \otimes Y f + Y u \otimes X f).  
\]
The lemma then follows from the action of the $\mathfrak{sl}_2$ generators on the basis vectors $f_{\pm}$:
\[
   H f_{\pm} = \pm f_{\pm}, \quad
   X f_+ = Y f_- = 0, \quad
   X f_- = f_+, \quad
   Y f_+ = f_-.  
\]
\end{proof}
\medskip
\subsection{Translation of Symmetry Breaking via Tensoring with ${\mathbb{C}}^3$}
\label{subsec:ts_C3}
~~~\newline
We continue our analysis of the symmetry breaking 
$\mathfrak{sl}(2, \mathbb{C}) \downarrow \mathfrak{gl}(1, \mathbb{C})$, 
extending the results of the previous section to the tensor product of $\Pi$
with the three-dimensional irreducible representation 
$F := \mathbb{C}^3$ of $\mathfrak{g}_{\mathbb{C}} = \mathfrak{sl}(2, \mathbb{C})$.

We retain Setting~\ref{set:sl_gl}. In particular,
${\mathfrak{g}}_{\mathbb{C}}=\operatorname{span}_{\mathbb C}\{H, X, Y\}$ and ${\mathfrak{g}}_{\mathbb{C}}'=\mathbb C H$.
We realize $F$ as the adjoint representation of $\mathfrak{g}_{\mathbb{C}}$.
Under the identification $F \simeq \mathfrak{g}_{\mathbb{C}}$, the $\mathfrak{g}'$-invariant subspace $F' \subset F$ corresponds to $\mathbb{C} H \simeq \mathbb{C}$.
The projection
\[
   \mathfrak{g}_{\mathbb{C}} = \mathbb{C} X \oplus \mathbb{C} H \oplus \mathbb{C} Y
   \longrightarrow \mathbb{C} H
\]
thus induces a natural projection
\[
   \operatorname{pr}_{F\to F'} \colon F \longrightarrow \mathbb{C}.
\]
Suppose that $\Pi \in {\mathcal{M}}(G)$ has infinitesimal character 
$\lambda \in \mathbb{C}$.
Then $\Pi \otimes \mathbb{C}^3$ admits the following primary decomposition:
\[
\Pi \otimes \mathbb{C}^3 \simeq 
P_{\lambda+2}(\Pi \otimes \mathbb{C}^3) \oplus 
P_{\lambda}(\Pi \otimes \mathbb{C}^3)
\oplus 
P_{\lambda-2}(\Pi \otimes \mathbb{C}^3).
\]
For $\delta \in \{1, 0, -1\}$, let
\begin{equation}
\operatorname{pr}_{\lambda \to \lambda+2\delta} \colon 
\Pi \otimes \mathbb{C}^3 \longrightarrow 
P_{\lambda+2\delta}(\Pi \otimes \mathbb{C}^3)
\end{equation}
denote the natural projection onto the corresponding primary component.

Arguing as in Theorem~\ref{thm:260314_mahoroba} for the case of $\Pi \otimes \mathbb{C}^2$, we obtain the following.
\begin{theorem}
\label{thm:26030325}
Assume that the center $\mathfrak{Z}(\mathfrak{g}_{\mathbb{C}})$ acts by scalars on each primary component 
$P_{\lambda + 2\delta}(\Pi \otimes \mathbb{C}^3)$ for $\delta \in \{0, \pm 1\}$.

Let $\Pi \in \mathcal{M}(G)$ have infinitesimal character 
$\lambda \in \mathbb{C}$ with respect to the Harish--Chandra isomorphism \eqref{eqn:HC_sl2}, 
and let $\delta \in \{1, -1\}$ be such that 
$\lambda \notin \{0, -\delta\}$.  

Let $T \colon \Pi \to \mathbb{C}_{\nu}$
be a symmetry breaking operator.  
Then, for any $u \in \Pi$, we have
\[
  (T \otimes \operatorname{pr}_{F\to F'}) 
  \circ 
  P_{\lambda + 2 \delta}(u \otimes H)
  =
  \frac{(\lambda + \delta)^2 - \nu^2}{2 \lambda (\lambda + \delta)} \, T u.
\]
\end{theorem}

\begin{proof}[Proof of Theorem~\ref{thm:26030325}]
The argument follows that of Theorem~\ref{thm:260314_mahoroba}, and we only sketch the main steps.

Define a polynomial in the Casimir operator $\Omega$ by
\[
  \phi_{\lambda}^{\lambda+2\delta} 
  := \frac{1}{64} (\Omega - \lambda^2 + 1)(\Omega - (\lambda - 2\delta)^2 + 1).
\]
Then $\phi_{\lambda}^{\lambda+2\delta}$ acts by the scalar 
$\frac{1}{2} \lambda (\lambda + \delta)$ on the primary component
$P_{\lambda+2\delta}(\Pi)$.  
Hence,
\[
   \phi_{\lambda}^{\lambda+2\delta} 
   = \frac{1}{2} \lambda (\lambda + \delta) \, \operatorname{pr}_{\lambda \to \lambda+2\delta}.
\]

Set $\widetilde{\Omega} := \Omega - \lambda^2 - 7$.  

A direct computation shows
\begin{align}
\label{eqn:26030320}
  (\operatorname{id} \otimes \operatorname{pr}_{F \to F'}) \, \widetilde{\Omega} (u \otimes H)
  &= 0,\\
\label{eqn:26030321}
  (\operatorname{id} \otimes \operatorname{pr}_{F \to F'}) \, \widetilde{\Omega}^2 (u \otimes H)
  &= 16 (\lambda^2 - 1) u \otimes H - 16 H^2 u \otimes H.
\end{align}

Writing
\[
   \phi_{\lambda}^{\lambda+2\delta} 
   = \frac{1}{64} (\widetilde{\Omega} + 8)(\widetilde{\Omega} + 4(1 + \delta \lambda)),
\]
we obtain
\[
  (\operatorname{id} \otimes \operatorname{pr}_{F \to F'}) \circ 
  \phi_{\lambda}^{\lambda+2\delta} (u \otimes H)
  = \frac{1}{4} (\lambda + \delta)^2 u \otimes H 
  - \frac{1}{4} H^2 u \otimes H.
\]

Applying the symmetry breaking operator $T$, we obtain
\[
  (T \otimes \operatorname{pr}_{F \to F'}) \circ 
  \phi_{\lambda}^{\lambda+2\delta} (u \otimes H)
  = \frac{1}{4} \bigl( (\lambda + \delta)^2 - \nu^2 \bigr) \, T u \otimes H.
\]
\end{proof}

\section{Symmetry Breaking for $\mathfrak{sl}(2,{\mathbb{C}}) \oplus \mathfrak{sl}(2,{\mathbb{C}}) \downarrow \mathfrak{sl}(2, {\mathbb{C}})$}
We now extend the translation principle for symmetry breaking operators to the setting of a pair $\widetilde G \supset G$  of real reductive groups 
whose complexified Lie algebras satisfy
\[
   (\widetilde{\mathfrak{g}}_{\mathbb{C}}, {\mathfrak{g}}_{\mathbb{C}})
   \simeq
   ({\mathfrak{sl}}(2, {\mathbb{C}}) \oplus {\mathfrak{sl}}(2, {\mathbb{C}}), 
    {\mathfrak{sl}}(2, {\mathbb{C}})).
\]

This framework naturally encompasses the fusion rules
 for tensor products of representations
 of ${\mathfrak{sl}}_2$, 
 and provides a uniform mechanism for describing
 how the associated multiplicities vary 
under translation 
 of infinitesimal characters.

Typical examples of such pairs $(\widetilde G, G)$ include the following,
covering compact, noncompact, and complex cases:
\begin{align*}
&(SU(2) \times SU(2), SU(2)),\\
&(SL(2, {\mathbb{R}}) \times SL(2, {\mathbb{R}}), SL(2, {\mathbb{R}})),\\
&(SL(2, {\mathbb{C}}), SU(2)),\\
&(SL(2, {\mathbb{C}}), SL(2, {\mathbb{R}})),\\
&(O(2,2), O(2,1)).
\end{align*}
The first two cases correspond precisely to the fusion rules
for tensor products of representations of $G$,
and serve as a guiding example for the general theory developed below.

Although the main results apply uniformly to all these settings,
we first present them in a form tailored to the study of fusion rules
(see Theorems~\ref{thm:26031537} and \ref{thm:26032032}).

In Section~\ref{sec:o4_o3}, 
we then reinterpret these results in a different geometric guise via the isomorphism
\[
  ({\mathfrak{sl}}(2, {\mathbb{C}}) \oplus {\mathfrak{sl}}(2,{\mathbb{C}}), 
   \operatorname{diag}({\mathfrak{sl}}(2, {\mathbb{C}})))
  \simeq
  ({\mathfrak{o}}(4, {\mathbb{C}}), {\mathfrak{o}}(3, {\mathbb{C}})),
\]
leading to Theorem~\ref{thm:26032009}.
\medskip
\subsection{Translating Fusion Rules}
~~~\newline
In this subsection, 
 we introduce a {\emph{translation principle}}
 for fusion rules
 associated with a simple Lie group $G$
 where complexified Lie algebra satisfies
 ${\mathfrak{g}}_{\mathbb{C}} \simeq {\mathfrak{sl}}(2,{\mathbb{C}})$.  
Our aim is to describe explicitly 
 how the multiplicities depend 
 on the parameters
 of the representations, 
 and how they transform
under translations
 of infinitesimal characters.

Explicit branching laws in this setting
 are classical in several important cases.
For finite-dimensional representations of $SU(2)$, they are given by the Pieri formula;
for irreducible unitary representations, they were determined independently 
by Molchanov \cite{Mol79} and Repka \cite{Rep79}.  
More recently, Kobayashi--Pevzner \cite{KP16} treated the fusion rules
 for non-unitary holomorphically induced representations, including negative parameters, 
and Clerc \cite{xCl17} studied non-unitary principal series representations.

In the non-unitary case, however,
 multiplicities may jump at certain parameter values,
 which highlights the need for a mechanism that controls these discontinuities.

The translation principle for fusion rules
 developed below provides an explicit and uniform criterion 
guaranteeing the stability of multiplicities, 
 culminating in Theorem~\ref{thm:26032032}.  

Let $\Pi'$ and $\Pi'' \in {\mathcal{M}}(G)$ have
${\mathfrak{Z}}(G)$-infinitesimal characters $\lambda'$ and $\lambda'' \in {\mathbb{C}}$, respectively.

As in \eqref{eqn:pr_pm}, let
\begin{equation}
\label{eqn:pr_pm_Pi}
 \operatorname{pr}_{\pm} \colon \Pi_{\lambda} \otimes {\mathbb{C}}^2
\to P_{\lambda\pm1}(\Pi_{\lambda} \otimes {\mathbb{C}}^2)
\end{equation}
denote the projections onto the corresponding primary components
for $\lambda=\lambda'$ and $\lambda''$.  

For $\delta, \varepsilon \in \{+, -\} \equiv \{1, -1\}$, 
we introduce the four projections
\[
   \operatorname{pr}_{\delta \varepsilon} \colon 
   (\Pi' \otimes {\mathbb{C}}^2) 
     \boxtimes 
   (\Pi'' \otimes {\mathbb{C}}^2)
    \to 
    P_{\lambda' + \delta}(\Pi' \otimes {\mathbb{C}}^2)
    \boxtimes
    P_{\lambda''+\varepsilon}(\Pi'' \otimes {\mathbb{C}}^2).
\]

Next, consider the projection $\operatorname{pr}_{F \to F'}$
associated with the fusion rule of the finite-dimensional representation 
of ${\mathfrak{sl}}(2,{\mathbb{C}})$:
\[
  F := {\mathbb{C}}^2 \otimes {\mathbb{C}}^2 
    \simeq {\mathbb{C}}^3 \oplus {\mathbb{C}}.
\]

Let $\{e_+, e_-\}$ and $\{f_+, f_-\}$
be the standard bases of the two copies of ${\mathbb{C}}^2$,  
so that 
\[
    {\mathbb{C}}^2 \otimes {\mathbb{C}}^2
    =
    \operatorname{Span}_{{\mathbb{C}}} \{e_+, e_-\}
    \otimes
    \operatorname{Span}_{{\mathbb{C}}} \{f_+, f_-\}.  
\]

Then the projection onto the trivial one-dimensional subrepresentation
\begin{equation}
\label{eqn:F_ef}
  F' := {\mathbb{C}} (e_+ \otimes f_- - e_- \otimes f_+) \simeq {\mathbb{C}}
\end{equation}
is given explicitly by 
\begin{equation}
\label{eqn:26022608}
   (a e_+ + b e_-) \otimes (c f_+ + d f_-) 
   \mapsto 
   \frac 1 2(ad-bc) (e_+ \otimes f_- -e_- \otimes f_+).   
\end{equation}
Finally, we define a linear map
\[
   H\colon \Pi' \otimes \Pi''
\to (\Pi' \otimes {\mathbb{C}}^2)
    \boxtimes
    (\Pi'' \otimes {\mathbb{C}}^2),
\] 
which realizes the embedding of $\Pi'\otimes\Pi''$ into the
$G$-submodule corresponding to the trivial component
$F' \simeq \mathbb C \subset \mathbb C^2 \otimes \mathbb C^2$
(i.e., the antisymmetric tensor component)
under the diagonal action.
Explicitly,
\begin{equation}
\label{eqn:H_uv}
   H(u,v) := 
   (u \otimes e_+) \otimes (v \otimes f_-) - 
   (u \otimes e_-) \otimes (v \otimes f_+),
\end{equation}
for $u \in \Pi'$ and $v \in \Pi''$.

With these definitions in place, 
the {\emph{translation of symmetry breaking}} from the pair of parameters
$(\lambda', \lambda'')$ to $(\lambda' + \delta, \lambda''+ \varepsilon)$
is explicitly controlled on the image of the map $H$, 
leading to the following foundational identity.
\begin{theorem}
\label{thm:26031537}
Let $G$ be a real reductive Lie group whose complexified Lie algebra satisfies
\[
\mathfrak g_{\mathbb C} \simeq \mathfrak{sl}(2,\mathbb C).
\]
Then there exist rational functions
\[
c_{\delta\varepsilon}(\lambda',\lambda'';\lambda'''),
\qquad (\delta,\varepsilon\in\{+,-\}),
\]
depending only on the parameters $\lambda',\lambda'',\lambda'''$, such that the
following holds.

Let $\Pi'$, $\Pi''$, and $\Pi'''$ be representations in $\mathcal M(G)$ with
infinitesimal characters $\lambda'$, $\lambda''$, and $\lambda'''$, respectively.
Assume that $\lambda',\lambda''\neq 0$, and that the primary components of
$\Pi'\otimes\mathbb C^2$ and $\Pi''\otimes\mathbb C^2$ are genuine eigenspaces of
$\mathfrak Z(\mathfrak g_{\mathbb C})$.

Let
\[
T\colon \Pi'\boxtimes \Pi'' \longrightarrow \Pi'''
\]
be a symmetry breaking operator, and let $\delta,\varepsilon\in\{+,-\}$.
Then, for every $u\in\Pi'$ and $v\in\Pi''$, we have
\begin{equation}
\label{eqn:26031537}
\bigl(T\otimes \operatorname{pr}_{F\to F'}\bigr)
\circ
\operatorname{pr}_{\delta\varepsilon}\bigl(H(u,v)\bigr)
=
c_{\delta\varepsilon}(\lambda',\lambda'';\lambda''')\, T(u\otimes v).
\end{equation}

Moreover, the functions $c_{\delta\varepsilon}(\lambda',\lambda'';\lambda''')$
are universal in the sense that they depend only on the infinitesimal characters
$\lambda'$, $\lambda''$, and $\lambda'''$, and are independent of the choice of
the real form $G$ and of the representations $\Pi'$, $\Pi''$, and $\Pi'''$.
Explicitly, they are given by
\begin{equation}
\label{eqn:c_formula}
c_{\delta\varepsilon}(\lambda',\lambda'';\lambda''')
=
\frac{
(\delta\lambda' + \varepsilon\lambda'' + 1 + \lambda''')
(\delta\lambda' + \varepsilon\lambda'' + 1 - \lambda''')
}{
8\,\lambda'\lambda''}.
\end{equation}
\end{theorem}
\medskip
\subsection{Outline of Proof}
~~~\newline
In a forthcoming paper \cite{xk26}, we prove the existence
 of such a rational function 
for more general pairs of Lie algebras 
$(\mathfrak{g}_{\mathbb C},\mathfrak{g}'_{\mathbb C})$, 
in the spirit of Theorems~\ref{thm:260314_mahoroba} and \ref{thm:26030325}.

The main task here is therefore to determine the rational function 
$c_{\delta\varepsilon}(\lambda',\lambda'';\lambda''')$. 

We first reduce the problem to the compact case, 
where explicit computations of tensor product decompositions are available. 

Since a rational function in the parameters 
$(\lambda',\lambda'',\lambda''')$ 
is determined by its values on a Zariski dense subset, 
it suffices to compute $c_{\delta\varepsilon}$ 
on a convenient Zariski dense set of parameters. 

The computation is carried out in the following subsections 
using only the finite-dimensional representation theory of compact groups.

\medskip
\subsection{Analogue of Rankin--Cohen Brackets in Finite-Dimensional Representations}
~~~\newline

Let $m \in \mathbb{N}$, and set
\[
V_m := \operatorname{Pol}_m[z] \simeq \mathbb{C}^{m+1},
\]
the space of polynomials in $z$ of degree at most $m$. 
This space carries a natural action of $SL(2,\mathbb{C})$ defined by
\[
  (\varpi_{-m}(g)F)(z) = (cz+d)^{m} F\!\left(\frac{az+b}{cz+d}\right),
  \qquad \text{for } g^{-1} = \begin{pmatrix} a & b \\ c & d \end{pmatrix}.
\]

The Pieri rule for tensor products implies that
\begin{equation}
\label{eqn:Pieri}
  [V_m \otimes V_n : V_k] \neq 0
  \quad \text{if and only if} \quad
  k = m+n-2\ell
\end{equation}
for some integer $\ell$ with $0 \le \ell \le \min(m,n)$.

To determine the rational function $c_{\delta \varepsilon}(\lambda_1, \lambda_2; \nu)$ in Theorem~\ref{thm:26031537}, 
we work with finite-dimensional representations. 
For this purpose, 
we provide an explicit symmetry breaking operator 
that will allow us to compute the coefficients 
$c_{\delta\varepsilon}(\lambda',\lambda'';\lambda''')$ explicitly.

\begin{proposition}
\label{prop:RC_fd}
Let $k = m + n - 2\ell$ with $0 \le \ell \le \min(m, n)$. 
Define
\[
  R_{m,n}^k \colon V_m \otimes V_n \to V_k
\]
by
\begin{align}
\label{eqn:25050114}
  R_{m,n}^k 
  := & \, \operatorname{Rest}_{z=w} \circ 
  \sum_{j=0}^{\ell} (-1)^{\ell-j} 
  \frac{(m+j-\ell)! (n-j)!}{j!(m-\ell)! (\ell-j)! (n-\ell)!} 
  \frac{\partial^{\ell}}{\partial z^{\ell-j} \partial w^j} \\
\notag
  = & \, \operatorname{Rest}_{z=w} \circ 
    \frac{(-1)^{\ell}n!}{\ell !(n-\ell)!}
  \left(\frac{\partial^{\ell}}{\partial z^{\ell}} 
        - \frac{(m+1-\ell)\ell}{n} 
        \frac{\partial^{\ell}}{\partial z^{\ell-1} \partial w} + \cdots \right).
\end{align}
Then $R_{m,n}^k$ is a nonzero symmetry breaking operator.
\end{proposition}

The operator $R_{m,n}^k$ provides a finite-dimensional model of the Rankin--Cohen bracket.
Its construction is obtained via meromorphic continuation of the Rankin--Cohen bidifferential operator, 
which we now recall.

Let $\mathcal{H} := \{z \in \mathbb{C} : \operatorname{Im} z > 0\}$ be the upper half-plane. 
For $\lambda \in \mathbb{Z}$, define a representation $\varpi_{\lambda}$ of 
$G_{\mathbb{R}} = SL(2,\mathbb{R})$ on the space of holomorphic functions on $\mathcal{H}$,
$\mathcal{O}(\mathcal{H})$, by
\[
  (\varpi_{\lambda}(g)F)(z)
  =(cz+d)^{-\lambda}F\!\left(\frac{az+b}{cz+d}\right),
  \qquad
\text{for } g^{-1}=\begin{pmatrix} a & b \\ c & d\end{pmatrix}.
\]

The representation $\varpi_{\lambda}$ has infinitesimal character $\lambda - 1$. 
It contains a discrete series representation if $\lambda \ge 2$, and a finite-dimensional subrepresentation of dimension $1-\lambda$ if $\lambda \le 0$.

For positive integers $k_1, k_2$ and $\ell \in \mathbb{N}$, the Rankin--Cohen bracket
\[
   \mathcal{R} \colon \varpi_{k_1} \otimes \varpi_{k_2} \to \varpi_{k_1+k_2+2\ell}
\]
is given by the differential operator
\[
  \sum_{j=0}^\ell (-1)^j 
  \binom{k_1+\ell-1}{j} 
  \binom{k_2+\ell-1}{\ell-j}
  \frac{\partial^{\ell}}{\partial z^{\ell-j}\partial w^{j}},
\]
followed by the restriction $\operatorname{Rest}_{z=w}$;
 see \cite{xCo75, xRa56}.

To extend this operator beyond the discrete series range, that is, to negative parameters,
we rewrite it in terms of infinitesimal characters and apply meromorphic continuation. 
Set $\lambda' := 1 - k_1$ and $\lambda'' := 1 - k_2$. 
Then the above operator can be expressed as
\[
  \sum_{j=0}^{\ell}
   \frac{(-1)^j}{j!(\ell-j)!} 
   \frac{\Gamma(1-\lambda'+\ell)}{\Gamma(1-\lambda'+\ell-j)}
   \frac{\Gamma(1-\lambda''+\ell)}{\Gamma(1-\lambda''+j)}
  \frac{\partial^{\ell}}{\partial z^{\ell-j}\partial w^{j}}.
\]

Using the reflection formula
\[
  \Gamma(z)\Gamma(1-z)=\frac{\pi}{\sin \pi z},
\]
this expression can be rewritten as
\begin{equation}
\label{eqn:RC_fd}
   \sum_{j=0}^{\ell}
   \frac{(-1)^j}{j!(\ell-j)!} 
   \frac{\Gamma(\lambda'-\ell+j)\Gamma(\lambda''-j)}
        {\Gamma(\lambda'-\ell)\Gamma(\lambda''-\ell)}
  \frac{\partial^{\ell}}{\partial z^{\ell-j}\partial w^{j}}.  
\end{equation}
This expression is well-suited for meromorphic continuation in the parameters $\lambda', \lambda''$.

\begin{proof}[Proof of Proposition~\ref{prop:RC_fd}]
Set $m := \lambda' - 1$ and $n := \lambda'' - 1$. 
Then \eqref{eqn:RC_fd} is well-defined for $\lambda',\lambda'' \in \mathbb{Z}_{>0}$ provided that
\[
   \ell \le \min(m,n).
\]

This expression, followed by the restriction $\operatorname{Rest}_{z=w}$, 
defines a symmetry breaking operator
\[
  \varpi_{1-\lambda'} \boxtimes \varpi_{1-\lambda''} \to \varpi_{1-\lambda'''},
\]
where $\lambda''' := \lambda' + \lambda'' - 2\ell - 1$. 

Moreover, this operator coincides with $R_{m,n}^k$ defined in \eqref{eqn:25050114}. 

Restricting to finite-dimensional subrepresentations yields a symmetry breaking operator
\[
  V_m \boxtimes V_n \to V_k,
\]
with $k = m+n-2\ell$.

The restricted operator is nonzero, since
\begin{equation}
\label{eqn:R_zl}
   R_{m,n}^k(z^{\ell} \otimes 1)=    \frac{(-1)^{\ell}n!}{\ell !(n-\ell)!}
 \neq 0.
\end{equation}
Hence the symmetry breaking operator $R_{m,n}^k$ generates the one-dimensional space
$\operatorname{Hom}_G(V_m \otimes V_n, V_k)$.
\end{proof}

\medskip
\subsection{Determining the Rational Function $c_{\delta \varepsilon}(\lambda', \lambda'';\lambda''')$}
~~~\newline
In this section, we explicitly determine the rational function 
$c_{\delta \varepsilon}(\lambda', \lambda'';\lambda''')$
appearing in Theorem~\ref{thm:26031537},
thereby completing the proof of the main formula.

\begin{theorem}
\label{thm:250502}
The function $c_{\delta \varepsilon}(\lambda', \lambda'';\lambda''')$ is given by
\[
   c_{\delta \varepsilon}(\lambda', \lambda'';\lambda''')
   =
   \frac{(\delta \lambda' + \varepsilon \lambda'' + 1 + \lambda''')
         (\delta \lambda' + \varepsilon \lambda'' + 1 - \lambda''')}
        {8 \lambda' \lambda''}.  
\]
\end{theorem}

To this end, 
we define the functions
\begin{align*}
g_{++}(m,n;k)&:= (m+n-k+2)(m+n+k+4), \\
g_{+-}(m,n;k)&:= (n-m+k)(m-n+k+2), \\
g_{-+}(m,n;k)&:= (m-n+k)(n-m+k+2), \\
g_{--}(m,n;k)&:= (m+n-k)(m+n+k+2).
\end{align*}

\begin{proposition}
\label{prop:25050220}
Let $m, n, k \in {\mathbb{N}}$.  
Suppose that 
$
   [V_m \otimes V_n: V_k] \ne 0.  
$
Then, setting  
\begin{align}
\label{eqn:lmd_mnk}
  (\lambda', \lambda'';\lambda''')&=(m+1,n+1;k+1),
\\
\intertext{we have}
\notag
  8 \lambda' \lambda'' \, c_{\delta\varepsilon}(\lambda', \lambda'';\lambda''')
  &=g_{\delta\varepsilon}(m,n;k).  
\end{align}
\end{proposition}
Since both sides are rational functions in 
$(\lambda',\lambda'';\lambda''')$,
and since the set of triples of infinitesimal characters
\begin{multline*}
  \{(m+1,n+1;k+1): [V_m \otimes V_n:V_k] \ne 0\}
\\
  = \{(m+1,n+1;k+1): m,n \in \mathbb{N},\ |m-n| \le k \le m+n,\ k\equiv m+n \!\!\!\pmod{2}\}
\end{multline*}
is Zariski dense in the $(\lambda',\lambda'';\lambda''')$-space, 
Theorem~\ref{thm:250502} follows immediately from Proposition~\ref{prop:25050220}.  
Consequently, 
the identity \eqref{eqn:26031537} in Theorem~\ref{thm:26031537} 
holds for arbitrary real forms $(G,G')$, 
with the rational function 
$c_{\delta \varepsilon}(\lambda', \lambda''; \lambda''')$ 
given by \eqref{eqn:c_formula}.

Retaining the notation in the previous subsection, 
suppose that $[V_m \otimes V_n:V_k] \ne 0$.  
Then $R_{m,n}^k \colon V_m \boxtimes V_n \to V_k$ defines a nonzero 
symmetry breaking operator.  

We now turn to the determination of the rational function 
$c_{\delta \varepsilon}(\lambda', \lambda'';\lambda''')$ 
using the finite-dimensional setting.

\begin{proposition} 
\label{prop:R_c} 
Let $\delta, \varepsilon \in \{+, -\}$, and define $H(u,v)$ as in \eqref{eqn:Hzw} for $u=z^{\ell}$ and $v=1$. Then
\[
  (R_{m,n}^k \otimes \operatorname{pr}_{F \to F'}) \circ \operatorname{pr}_{\delta \varepsilon}(H(u,v))
  = \frac{1}{8} \frac{g_{\delta \varepsilon}(m,n;k)}{(m+1)(n+1)} R_{m,n}^k \bigl(z^{\ell} \otimes 1\bigr).
\]
\end{proposition}
This proposition provides the explicit evaluation needed to identify 
the scalar $c_{\delta\varepsilon}(\lambda', \lambda''; \lambda''')$ in Theorem~\ref{thm:26031537}, showing Proposition~\ref{prop:25050220}. To prove Proposition~\ref{prop:R_c}, we analyze the terms appearing there (or in Theorem~\ref{thm:26031537}) one by one. We begin with the projections onto the primary components, as defined in \eqref{eqn:pr_pm_Pi}.

For the finite-dimensional representations $V_m$ with infinitesimal character $\lambda = m+1$, the irreducible decomposition (Pieri rule) is
\begin{equation}
\label{eqn:Vm2}
  V_m \otimes \mathbb{C}^2 \simeq V_{m+1} \oplus V_{m-1},
\end{equation}
as $\mathfrak{sl}_2$-modules.

The projections onto the irreducible summands
\[
  \operatorname{pr}_{\pm} \colon V_m \otimes \mathbb{C}^2 \to V_{m\pm 1} \subset V_m \otimes \mathbb{C}^2
\]
correspond to the projections onto the primary components.

The following lemma gives explicit formulas for these projections.

\begin{lemma}
\label{lem:25050109}
The decomposition \eqref{eqn:Vm2} can be realized explicitly as follows. For $F \in V_m=\operatorname{Pol}_m[z]$,
\begin{align}
\label{eqn:25050109}
  (m+1)\begin{pmatrix} F \\ 0 \end{pmatrix}
  =& \begin{pmatrix} (m+1 - z \frac{d}{d z})F \\[1mm] \frac{d}{d z} F \end{pmatrix} 
    + \begin{pmatrix} z\frac{d}{d z} F \\[1mm] -\frac{d}{d z} F \end{pmatrix}, \\
\notag
  (m+1)\begin{pmatrix} 0 \\ F \end{pmatrix}
  =& \begin{pmatrix} (m+1 - z \frac{d}{d z})(z F) \\[1mm] \frac{d}{d z} (z F) \end{pmatrix} 
    + \begin{pmatrix} z(z \frac{d}{d z}-m)F \\[1mm] -(z \frac{d}{d z}-m) F \end{pmatrix}.
\end{align}
\end{lemma}

The operator 
\[
  \operatorname{pr}_{\delta \varepsilon} \in \operatorname{End}((V_m \otimes \mathbb{C}^2) \otimes (V_n \otimes \mathbb{C}^2))
\]
is obtained by combining the projections
\[
  \operatorname{pr}_{\pm} \colon V_m \otimes \mathbb{C}^2 \to V_{m\pm 1}, 
  \qquad
  \operatorname{pr}_{\pm} \colon V_n \otimes \mathbb{C}^2 \to V_{n\pm 1}.
\]
Accordingly, we obtain four projections:
\begin{align*}
   \operatorname{pr}_{++} &: (V_m \otimes \mathbb{C}^2) \boxtimes (V_n \otimes \mathbb{C}^2) \to V_{m+1} \boxtimes V_{n+1}, \\
   \operatorname{pr}_{+-} &: (V_m \otimes \mathbb{C}^2) \boxtimes (V_n \otimes \mathbb{C}^2) \to V_{m+1} \boxtimes V_{n-1}, \\
   \operatorname{pr}_{-+} &: (V_m \otimes \mathbb{C}^2) \boxtimes (V_n \otimes \mathbb{C}^2) \to V_{m-1} \boxtimes V_{n+1}, \\
   \operatorname{pr}_{--} &: (V_m \otimes \mathbb{C}^2) \boxtimes (V_n \otimes \mathbb{C}^2) \to V_{m-1} \boxtimes V_{n-1}.
\end{align*}

Next, we compute the expression obtained by applying 
\[
  (\operatorname{id} \otimes \operatorname{pr}_{F \to F'}) \circ \operatorname{pr}_{\delta \varepsilon}
\]
to the element $H(u,v)$ defined in \eqref{eqn:H_uv}.

In the finite-dimensional setting, we take
\[
  u = F(z) \in V_m, 
  \qquad 
  v = G(w) \in V_n.
\]
Then $H(u,v)$ is a polynomial in $(z,w)$ with values in $\mathbb{C}^2 \otimes \mathbb{C}^2$, explicitly given by
\begin{equation}
\label{eqn:Hzw}
  H(u,v) := F(z)\, e_+ \otimes G(w)\, f_- - F(z)\, e_- \otimes G(w)\, f_+.
\end{equation}

On the other hand, the projection onto the trivial $G$-module,
\[
  \operatorname{pr}_{F \to F'} \colon \mathbb{C}^2 \otimes \mathbb{C}^2 \to \mathbb{C},
\]
is given explicitly in \eqref{eqn:F_ef}.

Thus, the computation reduces to evaluating each of the four cases corresponding to 
$(\delta,\varepsilon)\in\{+,-\}^2$:
\[
  (\operatorname{id} \otimes \operatorname{pr}_{F \to F'}) \circ \operatorname{pr}_{\delta \varepsilon}(H(z,w)),
\]
which we now carry out explicitly.
\begin{lemma}
Let $F \in V_m$ and $G \in V_n$. Then 
\[
  2(m+1)(n+1) (\operatorname{id} \otimes \operatorname{pr}_{F \to F'}) \circ \operatorname{pr}_{\delta \varepsilon}(H(z,w))
\]
is given by:
\begin{itemize}
\item For $(\delta, \varepsilon) = (+,+)$:
\begin{equation}
\label{eqn:25050409}
  (z-w)^2 F' G' + (z-w)(-m F G' + n F' G) + (m+n+2) F G,
\end{equation}
\item For $(\delta, \varepsilon) = (+,-)$:
\begin{equation}
\label{eqn:26030121}
  -(z-w)^2 F' G' + (z-w)(m F G' - n F' G) + (m+1)n F G,
\end{equation}
\item For $(\delta, \varepsilon) = (-,+)$:
\begin{equation}
\label{eqn:26030140}
  -(z-w)^2 F' G' + (z-w)(m F G' - n F' G) + m(n+1) F G,
\end{equation}
\item For $(\delta, \varepsilon) = (-,-)$:
\begin{equation}
\label{eqn:25050209}
  (w-z)^2 F' G' + (w-z)(-n F' G + m F G').
\end{equation}
\end{itemize}
\end{lemma}

\begin{proof}
We illustrate the case $(\delta, \varepsilon) = (+,+)$. By \eqref{eqn:26022608} and Lemma~\ref{lem:25050109}, the expression can be written as
\[
  \det \begin{pmatrix} (m+1)F - zF' & n w G - w^2 G' \\[1mm]
                        F' & G + w G'
       \end{pmatrix}
  -
  \det \begin{pmatrix} mzF - z^2 F' & (n+1)G - w G' \\[1mm]
                        F + zF' & G'
       \end{pmatrix}.
\]
A straightforward computation yields the formula given in \eqref{eqn:25050409}. 
The remaining cases follow by analogous computations.
\end{proof}
\begin{proof}[Proof of Proposition~\ref{prop:R_c}]
Let 
\[
  T = R_{m,n}^k \colon V_m \boxtimes V_n \to V_k
\]
be the symmetry breaking operator. Then it induces an $\mathfrak{sl}_2$-homomorphism
\[
  R_{m,n}^k \otimes \operatorname{pr}_{F \to F'} \colon 
  (V_m \otimes \mathbb{C}^2) \boxtimes (V_n \otimes \mathbb{C}^2) \to V_k \otimes \mathbb{C} \simeq V_k.
\]

We first consider the case $(\delta, \varepsilon) = (+,+)$ with $u = z^{\ell}$ and $v = 1$.

By \eqref{eqn:25050409}, we have
\begin{align*}
  & 2(m+1)(n+1) (R_{m,n}^k \otimes \operatorname{pr}_{F \to F'}) \circ \operatorname{pr}_{++}(H(u,v)) \\
  =\;& R_{m,n}^k \bigl( n \ell (z-w) z^{\ell-1} + (m+n+2) z^{\ell} \bigr).
\end{align*}

Applying the explicit formula \eqref{eqn:25050114} for $R_{m,n}^k$, this becomes

\begin{align*}
  & \frac{(-1)^\ell n!}{\ell!(n-\ell)!} 
    \Biggl(
      \frac{\partial^\ell}{\partial z^\ell} 
      - \frac{(m+1-\ell)\ell}{n} 
        \frac{\partial^\ell}{\partial z^{\ell-1} \partial w} +\dotsb
    \Biggr) \Big|_{w=z} \\
  & \qquad 
    \bigl( n \ell (z-w) z^{\ell-1} + (m+n+2) z^{\ell} \bigr) \\
  =\;& \frac{(-1)^\ell n!}{(n-\ell)!} (\ell+1)(m+n-\ell+2).
\end{align*}

By the definition of $g_{++}(m,n;k)$ with $k = m+n-2\ell$, this simplifies to
\[
  \frac{(-1)^\ell n!}{4 (n-\ell)!} g_{++}(m,n;k).
\]

On the other hand, from \eqref{eqn:R_zl},
\[
  R_{m,n}^k(z^\ell \otimes 1) = \frac{(-1)^\ell n!}{(n-\ell)!}.
\]

Combining these, we obtain
\begin{equation}
\label{eqn:26031906}
  (R_{m,n}^k \otimes \operatorname{pr}_{F \to F'}) \circ \operatorname{pr}_{++}(H(z^\ell,1))
  = \frac{(\ell+1)(m+n-\ell+2)}{2 (m+1)(n+1)} R_{m,n}^k(z^\ell \otimes 1).
\end{equation}

Hence, the constant $c_{++}(\lambda',\lambda'';\lambda''')$ in \eqref{eqn:26031537} is
\[
  c_{++}(\lambda',\lambda'';\lambda''') = \frac{g_{++}(m,n;k)}{8 (m+1)(n+1)}.
\]

Rewriting the right-hand side in terms of infinitesimal characters, as in \eqref{eqn:lmd_mnk}, yields the desired formula for $c_{++}(\lambda_1,\lambda_2;\nu)$.

The other cases $(+,-)$, $(-,+)$, and $(-,-)$ are treated in an analogous manner. 
Thus the explicit formula \eqref{eqn:c_formula} is obtained.
This completes the proof.
\end{proof}

\medskip 
\subsection{Stability Theorem of Multiplicities in Fusion Rule}
\label{subsec:stability_f}
~~~\newline

In this subsection, we formulate a stability theorem 
for multiplicities in fusion rules, 
describing regions in the parameter space 
on which the multiplicity remains constant.

For $\xi', \xi'' \ne0$, 
 we set
\[
   \Lambda(\xi', \xi''):=\Lambda(\xi') \times \Lambda(\xi''), 
\]
where $\Lambda(\xi) \subset \xi + {\mathbb{Z}}$ 
 is defined in \eqref{eqn:dom_xi}.  

Set
\[
  ({\mathbb{Z}}^2)_{\operatorname{even}}
  :=
  \{(k_1, k_2) \in {\mathbb{Z}}^2:
     k_1 \equiv k_2 \operatorname{mod} 2
  \}.  
\]
We define $\Xi$ by
\[
   \Xi \equiv \Xi(\xi', \xi'')
       :=\Lambda(\xi', \xi'') \cap ((\xi', \xi'')+({\mathbb{Z}}^2)_{\operatorname{even}}), 
\]
which serves as the natural parameter lattice compatible with parity conditions in the tensor product.

\begin{definition}
[Fence]
\label{def:fence}
Let $\xi', \xi''\in {\mathbb{C}} \setminus \{0\}$
 and $\mu \in {\mathbb{C}}$.  
For each $\alpha$, $\beta \in \{+, -\}$, 
 we define a subset
 ${\mathcal{D}}_{\alpha\beta} \equiv {\mathcal{D}}_{\alpha\beta}(\xi', \xi'', \mu)$
 of $\Xi \equiv \Xi(\xi', \xi'')$
 by 
\par\noindent
{\bf{Case 1.}}\enspace
If $\alpha \xi' + \beta \xi'' + \mu \not \in 2{\mathbb{Z}}+1$, 
 we set ${\mathcal{D}}_{\alpha\beta}:=\Xi$, 
\par\noindent
{\bf{Case 2.}}\enspace
If $\alpha \xi' + \beta \xi'' + \mu \in 2{\mathbb{Z}}+1$, 
\[
   {\mathcal{D}}_{\alpha\beta}
   =
   \{(\lambda', \lambda'') \in \Xi:
     \operatorname{sgn}(\alpha \lambda' + \beta \lambda''+\mu)
    =
     \operatorname{sgn}(\alpha \xi' + \beta \xi''+\mu)
   \}.  
\]
We set
\[
   {\mathcal{D}}(\xi', \xi'';\mu)
   :=
   \bigcap_{\alpha, \beta\in \{+, -\}}
   {\mathcal{D}}_{\alpha\beta}.  
\]
\end{definition}
The boundaries of the sets ${\mathcal{D}}_{\alpha\beta}$ may be regarded as 
``fences'' in the parameter space across which the multiplicity may jump; 
the region enclosed by these fences defines a domain where no such jumps occur.

\begin{example}
\label{ex:fence_mnk}
Fix $\mu \in {\mathbb{N}}_+$.  
Then there are five possibilities
 of the region ${\mathcal{D}}(\xi', \xi'';\mu)$
 for $\xi', \xi'' \in {\mathbb{N}}_+$, 
 and it gives a decomposition
 of ${\mathbb{N}}_+^2$:
\begin{align*}
 {\mathbb{N}}_+^2=\,&
({\mathbb{N}}_+^2)_{\operatorname{even}} \amalg ({\mathbb{N}}_+^2)_{\operatorname{odd}}
\\
 =\,&({\mathbb{N}}_+^2)_{\operatorname{even}}
  \amalg
  \bigcup_{s \in \{\uparrow, \nearrow, \rightarrow, \swarrow \}}
  ({\mathbb{N}}_+^2)_{\operatorname{odd}}^s, 
\end{align*}
where
\begin{align*}
 ({\mathbb{N}}_+^2)_{\operatorname{even}}
  :=\,&\{(\lambda',\lambda''):
          \lambda',\lambda'' \in {\mathbb{N}}_+,\,
          \lambda'+\lambda''+\lambda''' \in 2{\mathbb{Z}}\}, 
\\
  ({\mathbb{N}}_+^2)_{\operatorname{odd}}
  :=\,&\{(\lambda',\lambda''):
          \lambda',\lambda'' \in {\mathbb{N}}_+,\,
          \lambda'+\lambda''+\lambda''' \in 2{\mathbb{Z}}+1\}, 
\\
 ({\mathbb{N}}_+^2)_{\operatorname{odd}}^{\uparrow}
  :=\,&\{(\lambda',\lambda'')\in  ({\mathbb{N}}_+^2)_{\operatorname{odd}}:
          \lambda''- \lambda' \ge \lambda'''+1\}, 
\\
 ({\mathbb{N}}_+^2)_{\operatorname{odd}}^{\nearrow}
  :=\,&\{(\lambda',\lambda'')\in  ({\mathbb{N}}_+^2)_{\operatorname{odd}}:
        |\lambda'-\lambda''|\le \lambda'''-1, \lambda'+\lambda'' \ge \lambda'''+1\}, 
\\
 ({\mathbb{N}}_+^2)_{\operatorname{odd}}^{\rightarrow}
  :=\,&\{(\lambda',\lambda'')\in  ({\mathbb{N}}_+^2)_{\operatorname{odd}}:
          \lambda'- \lambda'' \ge \lambda'''+1\}, 
\\
 ({\mathbb{N}}_+^2)_{\operatorname{odd}}^{\swarrow}
  :=\,&\{(\lambda',\lambda'')\in  ({\mathbb{N}}_+^2)_{\operatorname{odd}}:
        |\lambda'-\lambda''| \le \lambda'''-1, \lambda'+\lambda''\le \lambda'''-1\}.  
\end{align*}
This decomposition illustrates how the stability region 
is divided into finitely many chambers.
\end{example}

As in the stability theorem for the branching corresponding to $\mathfrak{sl}(2, \mathbb C) \downarrow \mathfrak{gl}(1, \mathbb C)$ in Theorem~\ref{thm:26031818}, 
we deduce the following stability result 
as a direct consequence of Theorem~\ref{thm:26031537}.

\begin{theorem}[Stability of Multiplicities]
\label{thm:26032032}
Let $G$ be a simple Lie group whose complexified Lie algebra satisfies
\[
\mathfrak g_{\mathbb C}=\mathfrak{sl}(2,\mathbb C).
\]
Let
\[
\{\Pi'_{\lambda'} : \lambda' \in \Lambda(\xi')\},
\qquad
\{\Pi''_{\lambda''} : \lambda'' \in \Lambda(\xi'')\}
\subset \mathcal M(G)
\]
be reduced coherent families of representations passing through
$\Pi'=\Pi'_{\xi'}$ and $\Pi''=\Pi''_{\xi''}$, respectively, where
$\xi'$ and $\xi''$ are nonsingular infinitesimal characters.
Let $\Pi_{\lambda'''} \in \mathcal M(G)$ be a representation with
infinitesimal character $\lambda'''$.

Then the multiplicity
\[
\bigl[\Pi'_{\lambda'} \boxtimes \Pi''_{\lambda''} : \Pi_{\lambda'''}\bigr]
\]
is constant for all parameters
\[
(\lambda',\lambda'') \in \mathcal D(\xi',\xi'';\lambda''').
\]
In particular, the multiplicity is locally constant within the region
$\mathcal D(\xi',\xi'';\lambda''')$ of the parameter space, and may change
only when crossing its boundary.
\end{theorem}

In particular, multiplicities are constant inside the region ${\mathcal{D}}(\xi', \xi'';\lambda''')$ 
and may change only when crossing its boundaries.

\medskip 
\subsection{Examples for Stable Multiplicities in Fusion Rule}
\label{subsec:stability_ex}
~~~\newline
We present illustrative examples of the stability of multiplicities discussed in
Theorem~\ref{thm:26031818}.
\begin{example}
\label{ex:fence_Pieri}
We consider the multiplicity 
\[
[V_m \otimes V_n:V_k]
\]
 for the tensor product of finite-dimensional representations.  
The infinitesimal characters
 of $V_m$, $V_n$, $V_k$ are given 
by $\lambda'=m+1$, $\lambda''=n+1$, 
and $\lambda'''=k+1$, respectively.

Fixing $k$ and applying Theorem~\ref{thm:26032032} 
 together with Example~\ref{ex:fence_mnk}, 
 the multiplicity $[V_m \otimes V_n:V_k]$ is constant on each of the following $(m,n)$-parameter regions:
\begin{align*}
&\bullet \enspace
m+n \not \equiv k \operatorname{mod}2, 
\\
&\bullet \enspace
m+n \equiv k \operatorname{mod}2,\, n-m \ge k+2, 
\\
&\bullet \enspace
m+n \equiv k \operatorname{mod}2,\, |m-n| \le k,\, m+n \ge k, 
\\
&\bullet \enspace
m+n \equiv k \operatorname{mod}2,\, m-n \ge k+2,
\\
&\bullet \enspace
m+n \equiv k \operatorname{mod}2,\, |m-n| \le k,\, |m+n| \le k-2.  
\end{align*}
This fits with the Pieri rule asserting that
\[
  [V_m \otimes V_n:V_k]=1
\,\,\Leftrightarrow\,\,
  m+n \equiv k \operatorname{mod}2,\, |m-n| \le k,\, m+n \ge k.  
\]
\end{example}

\begin{example}
\label{ex:26032029}
Let $\mathfrak g_{\mathbb C}=\mathfrak{sl}(2,\mathbb C)$, and let
$\mathfrak b_{\mathbb C}=\mathbb{R}H+\mathbb{R}X\subset\mathfrak g_{\mathbb C}$
be a Borel subalgebra.
For $\mu\in\mathbb C$, let
\[
M(\mu):=
U(\mathfrak g_{\mathbb C})
\otimes_{U(\mathfrak b_{\mathbb C})}
\mathbb C_\mu
\]
denote the corresponding Verma module.
In our convention, $M(\mu)$ has infinitesimal character $\lambda=\mu+1$,
and admits the finite-dimensional irreducible quotient $V_m$
when $\mu=m\in\mathbb N$.

We consider the multiplicity
\[
m(\mu',\mu'';\mu''')
:=
\dim\operatorname{Hom}_{\mathfrak g_{\mathbb C}}
\bigl(
M(\mu'''),\,
M(\mu')\otimes M(\mu'')
\bigr)
\]
for parameters $\mu',\mu'',\mu'''\in\mathbb C$.
For generic triples $(\mu',\mu'',\mu''')\in\mathbb C^3$,
this multiplicity is either $0$ or $1$;
see \cite[Theorem~5.1]{K12} for the multiplicity-free theorem
in the general setting of restricting parabolic Verma modules
with respect to reductive symmetric pairs.

A remarkable multiplicity-two phenomenon was discovered in \cite{KP16}:
\[
m(\mu',\mu'';\mu''')=2
\]
if and only if
\[
\mu',\mu'',\mu'''\in\mathbb Z,\quad
\mu'+\mu''-\mu'''\in 2\mathbb N,\quad
\mu'+\mu''+\mu'''\le -2,\quad
\mu'-\mu''\le -\mu'''-2.
\]
In terms of infinitesimal characters
\[
\lambda'=\mu'+1,\qquad
\lambda''=\mu''+1,\qquad
\lambda'''=\mu'''+1,
\]
these conditions are equivalent to
\[
\lambda',\lambda'',\lambda'''\in\mathbb Z,\qquad
\lambda'+\lambda''-\lambda'''\in 2\mathbb N+1,
\]
together with
\[
\lambda'+\lambda''+\lambda'''\le 1,
\qquad
\lambda'-\lambda''\le -\lambda'''-1.
\]
Once again, these inequalities describe precisely one of the regions
identified in Theorem~\ref{thm:26032032}
on which the multiplicity is stable;
see also Example~\ref{ex:fence_mnk}.
\end{example}

\section{
Symmetry Breaking for ${\mathfrak{o}}(4, \mathbb C) \downarrow {\mathfrak{o}}(3, \mathbb C)$}
\label{sec:o4_o3}
Since the results in Theorem~\ref{thm:26031537} depend only on the pair of complexified Lie algebras 
$({\mathfrak{g}}_{\mathbb{C}}, {\mathfrak{g}}'_{\mathbb{C}})$, 
they admit an equivalent reformulation via the isomorphism
\begin{equation}
\label{eqn:sl2_o3}
  ({\mathfrak{sl}}(2, {\mathbb{C}}) \oplus {\mathfrak{sl}}(2,{\mathbb{C}}), 
   \operatorname{diag}({\mathfrak{sl}}(2, {\mathbb{C}})))
  \simeq
  ({\mathfrak{o}}(4, {\mathbb{C}}), {\mathfrak{o}}(3, {\mathbb{C}})).  
\end{equation}
Fix a standard basis of a Cartan subalgebra of ${\mathfrak{o}}(n,{\mathbb{C}})$ for $n=3,4$.
The Harish-Chandra isomorphisms are then given by
\begin{align*}
  \operatorname{Hom}_{{\mathbb{C}}\operatorname{-alg}}
  ({\mathfrak{Z}}({\mathfrak{o}}(4,{\mathbb{C}})), {\mathbb{C}})
  \simeq\,
  & {\mathbb{C}}^2/W(D_2), 
\\
  \operatorname{Hom}_{{\mathbb{C}}\operatorname{-alg}}
  ({\mathfrak{Z}}({\mathfrak{o}}(3,{\mathbb{C}})), {\mathbb{C}})
  \simeq\,
  & {\mathbb{C}}/W(B_1).  
\end{align*}

Let $F:={\mathbb{C}}^4$ denote
 the natural representation of ${\mathfrak{g}}_{\mathbb{C}}={\mathfrak{o}}(4, {\mathbb{C}})$.  
As a representation of ${\mathfrak{g}}_{\mathbb{C}}'$, 
 it decomposes
 as ${\mathbb{C}}^4={\mathbb{C}} \oplus {\mathbb{C}}^3 =:F'\oplus F''$.  
We fix a generator
 $f_0 \in F'$
 and identify $F'={\mathbb{C}}f_0$ with ${\mathbb{C}}$.  
Let 
\[
  \operatorname{pr}_{F} \colon F \to F' \simeq {\mathbb{C}}
\]
 denote the projection.  

\begin{theorem}
\label{thm:26032009}
Let $(G,G')$ be a pair of real reductive Lie groups whose complexified Lie algebras satisfy
\[
(\mathfrak g_{\mathbb C},\mathfrak g'_{\mathbb C})
\simeq
(\mathfrak{o}(4,\mathbb{C}),\mathfrak{o}(3,\mathbb{C})).
\]
Then there exist rational functions
\[
C_{i,\delta}(\lambda_1,\lambda_2;\nu),
\qquad
i\in\{1,2\},\ \delta\in\{+,-\},
\]
depending only on the parameters $\lambda_1,\lambda_2,$ and $\nu$, such that the following holds.

Let $\Pi_{\lambda}\in\mathcal{M}(G)$ be a representation with infinitesimal characters
$(\lambda_1,\lambda_2)$, and let $\pi_{\nu}\in\mathcal{M}(G')$ have infinitesimal character
$\nu\in\mathbb{C}$.
Assume that $\lambda_1\neq \pm\lambda_2$, and that the primary components of
$\Pi_{\lambda}\otimes\mathbb C^{4}$ are genuine eigenspaces of ${\mathfrak Z}(G)=U(\mathfrak{g}_\mathbb C)^G$.

Let
\[
T:\Pi_{\lambda}\longrightarrow\pi_{\nu}
\]
be a symmetry breaking operator.
Fix $i\in\{1,2\}$ and $\delta\in\{+,-\}$.
Then, for every $u\in\Pi_{\lambda}$, we have
\[
(T\otimes\operatorname{pr}_{F\to F'})
\circ
\operatorname{pr}_{\lambda\to\lambda+\delta e_i}(u\otimes f_0)
=
C_{i,\delta}(\lambda_1,\lambda_2;\nu)\,Tu.
\]

Moreover, the rational functions $C_{i,\delta}(\lambda_1,\lambda_2;\nu)$ are
\emph{universal} in the sense that they depend only on the infinitesimal characters
$(\lambda_1,\lambda_2)$ and $\nu$, and are independent of the choice of real forms
$(G,G')$ as well as of the representations $\Pi_{\lambda}$ and $\pi_{\nu}$.
Explicitly, they are given by
\begin{equation}
\label{eqn:26032009}
C_{i,\delta}(\lambda_1,\lambda_2;\nu)
=
\frac{
\bigl(\lambda_i+\nu+\tfrac12\delta\bigr)
\bigl(\lambda_i-\nu+\tfrac12\delta\bigr)
}{
2(\lambda_1^2-\lambda_2^2)
}.
\end{equation}
\end{theorem}

\begin{proof}
This is a reformulation of Theorem~\ref{thm:26031537}, 
 and we only outline the proof.

Via the isomorphism \eqref{eqn:sl2_o3}, 
 the infinitesimal characters
 $(\lambda', \lambda'')$ and $\lambda'''$
 for $({\mathfrak{sl}}_2 \oplus {\mathfrak{sl}}_2, \operatorname{diag}({\mathfrak{sl}}_2))$
 correspond to $(\lambda_1,\lambda_2)$ and $\nu$ by
\[
   \lambda'=\lambda_1+\lambda_2,\quad
   \lambda''=\lambda_1-\lambda_2,\quad
   \lambda'''=2\nu.  
\]
Moreover, the projections $\operatorname{pr}_{\delta\varepsilon}$
 correspond to $\operatorname{pr}_{i, \kappa}$
 according to the following correspondence:
\begin{table}[H]
\begin{center}
\begin{tabular}{c|cccc}
$\delta \varepsilon$
&$++$\,\,
&$+-$\,\,
&$-+$\,\,
&$--$
\\
\hline
$(i, \kappa)$
&$1+$
&$2+$
&$2-$
&$1-$
\\
\end{tabular}
\end{center}
\end{table}

Hence, 
 the rational functions
 $C_{i, \delta}(\lambda_1,\lambda_2;\nu)$
 in Theorem~\ref{thm:26032009} are obtained from $c_{\delta \varepsilon}(\lambda', \lambda'';\lambda''')$
 in Theorem~\ref{thm:26031537}
 via the relation
\[
   C_{i, \delta}(\lambda_1,\lambda_2;\nu)
   =
   c_{\delta \varepsilon}(\lambda', \lambda'';\lambda''').  
\]
This completes the proof of the theorem.
\end{proof}
This reformulation highlights the intrinsic nature of the translation principle,
showing that it depends only on the underlying complexified Lie algebras
and not on the choice of real forms.

\end{document}